\documentclass{lmcs}
\pdfoutput=1

\usepackage[utf8]{inputenc}

\usepackage{lastpage}
\lmcsdoi{19}{4}{26}
\lmcsheading{}{\pageref{LastPage}}{}{}%
{Sep.~14,~2022}{Dec.~18,~2023}{}

\usepackage{stackengine} 
\usepackage{multicol}
\usepackage[shortlabels]{enumitem}
\newcommand{\quot}[2]{#1/#2}

\newcommand{\quotpi}[1]{[#1]}

\newcommand{\moment}{{\mathfrak m}}

\newcommand{\noment}{{\mathfrak n}}

\newcommand{\dn}{\square}
\newcommand{\dd}{\lozenge}

\usepackage{thmtools}
\usepackage{amssymb}
\usepackage{amsfonts}
\usepackage{tikz}
\usepackage{apxproof}
\usepackage{float}
\usepackage{tikz-cd}
\usepackage{caption}
\usepackage[font={small,it}]{caption}
\usepackage{siunitx}
\usepackage{amsmath}
\usepackage{amsthm}
\usepackage{centernot}
\usepackage{stmaryrd}
\usepackage{hyperref}

\usepackage{wasysym}

\theoremstyle{plain}% default

\NewDocumentCommand{\bc}{}{{\text{\normalfont\fullmoon}}}

\newcommand{\lb}{\langle}
\newcommand{\rb}{\rangle}

\newcommand\loz{\stackMath\mathbin{\stackinset{c}{0ex}{c}{0ex}{\ast}{\lozenge}}}
\newcommand{\gog}{\mathfrak S}

\title{Dynamic Cantor Derivative Logic} %TODO Please add

\begin{document}

\author[D.~Fern\'andez-Duque]{David Fern\'andez-Duque\lmcsorcid{0000-0001-8604-4183}}[a]
\author[Y.~Montacute]{Yo\`av Montacute\lmcsorcid{0000-0001-9814-7323}}[b] 

%\author[$\mathsection$]{David Fern\'andez-Duque}
%\author[$\dagger$]{Yo\`av Montacute}

\address{University of Barcelona, Spain}
%Carrer de Montalegre 6-8, 08001\\}	%optional
% write emails for all authors having that affiliation
\email{fernandez-duque@ub.edu}  %optional
% affiliation 2 (automatically numbered b)
\address{University of Cambridge, United Kingdom}	%optional
\email{yoav.montacute@cl.cam.ac.uk}  %optional


\begin{abstract}
Topological semantics for modal logic based on the Cantor derivative operator gives rise to derivative logics, also referred to as $d$-logics.
Unlike logics based on the topological closure operator, $d$-logics have not previously been studied in the framework of dynamical systems, which are pairs $(X,f)$ consisting of a topological space $X$ equipped with a continuous function $f\colon X\to X$.

We introduce the logics $\bf{wK4C}$, $\bf{K4C}$ and $\bf{GLC}$ and show that they all have the finite Kripke model property and are sound and complete with respect to the $d$-semantics in this dynamical setting. In particular, we prove that $\bf{wK4C}$ is the $d$-logic of all dynamic topological systems, $\bf{K4C}$ is the $d$-logic of all $T_D$ dynamic topological systems, and $\bf{GLC}$ is the $d$-logic of all dynamic topological systems based on a scattered space.
We also prove a general result for the case where $f$ is a homeomorphism, which in particular yields soundness and completeness for the corresponding systems $\bf{wK4H}$, $\bf{K4H}$ and $\bf{GLH}$.

The main contribution of this work is the foundation of a general proof method for finite model property and completeness of dynamic topological $d$-logics. Furthermore, our result for $\bf{GLC}$ constitutes the first step towards a proof of completeness for the trimodal topo-temporal language with respect to a finite axiomatisation -- something known to be impossible over the class of all spaces.
\end{abstract}

\maketitle

\section{Introduction}

Dynamic (topological) systems are mathematical models of processes that may be iterated indefinitely.
Formally, they are defined as pairs $\lb\mathfrak X,f\rb$ consisting of a topological space $\mathfrak{X}=\langle X,\tau\rangle$ and a continuous function $f\colon X\to X$; the intuition is that points in the space $\mathfrak X$ `move' along their orbit, $x,f(x),f^2(x),\ldots$ which usually simulates changes in time.
{\em Dynamic topological logic} ($\bf DTL$) combines modal logic and its topological semantics with linear temporal logic (see Pnueli \cite{ltl}) in order to reason about dynamical systems in a decidable framework.

Due to their rather broad definition, dynamical systems are routinely used in many pure and applied sciences, including computer science.
To cite some recent examples, in data-driven dynamical systems, data-related problems may be solved through data-oriented research in dynamical systems as suggested by Brunton and Kutz \cite{brunton2019data}. 
Weinan~\cite{weinan2017proposal} proposes a dynamic theoretic approach to machine learning where dynamical systems are used to model nonlinear functions employed in machine learning. 
Lin and Antsaklis's~\cite{hybrid} hybrid dynamical systems have been at the centre of research in control theory,  artificial intelligence and computer-aided verification. 
Mortveit and Reidys's~\cite{sequential} sequential dynamical systems generalise aspects of systems such as cellular automata, and also provide a framework through which we can study dynamical processes in graphs.
Another example of dynamical systems and computer science can be found in the form of linear dynamical systems, i.e.\ systems with dynamics given by a linear transformation. 
Examples of such systems in computer science include Markov chains, linear recurrence sequences and linear differential equations. 
Moreover, there are known strong connections between dynamical systems and algorithms. This may be found for example in the work of Hanrot, Pujol and Stehl\'e~\cite{hanrot2011analyzing}, and in the work of Chu~\cite{chu2008linear}.

Such applications raise a need for effective formal reasoning about topological dynamics.
Here, we may take a cue from modal logic and its topological semantics.
The study of the latter dates back to McKinsey and Tarski~\cite{Tarski}, who proved that the modal logic $\bf S4$ is complete for a wide class of spaces, including the real line.
Artemov, Davoren and Nerode~\cite{artemov} extended $\bf S4$ with a `next' operator in the spirit of $\bf LTL$, producing the logic $\bold{S4C}$. 
They proved that this logic is sound and complete with respect to the class of all dynamic topological systems.
The system $\bold{S4C}$ was enriched with the `henceforth' tense by Kremer and Mints, who dubbed the new logic {\em dynamic topological logic} ($\bold{DTL}$). 
Later, Konev et al.\ \cite{konev} showed that $\bold{DTL}$ is undecidable, and Fernández-Duque~\cite{david} showed that it is not finitely axiomatisable on the class of all dynamic topological systems.

The aforementioned work on dynamic topological logic interprets the modal operator $\lozenge$ as a closure operator.
However, McKinsey and Tarski had already contemplated semantics that are instead based on the Cantor derivative \cite{Tarski}: the {\em Cantor derivative} of a set $A$, usually denoted by $d(A)$, is the set of points $x$ such that $x$ is in the closure of $A\setminus \{x\}$ (see Section~\ref{secPrel}).
This interpretation is often called {\em $d$-semantics} and the resulting logics are called {\em $d$-logics.}
These logics were first studied in detail by Esakia, who showed that the $d$-logic $\bold{wK4}$ is sound and complete with respect to the class of all topological spaces \cite{Esakia2}.
It is well-known that semantics based on the Cantor derivative are more expressive than semantics based on the topological closure.
For example, consider the property of a space $\mathfrak X$ being {\em dense-in-itself}, meaning that $\mathfrak X$ has no isolated points (see Section~\ref{secTangle}).
The property of being dense-in-itself cannot be expressed in terms of the closure operator, but it {\em can} be expressed in topological $d$-semantics by the formula $\lozenge\top$.

Logics based on the Cantor derivative appear to be a natural choice for reasoning about dynamical systems. However, there are no established results of completeness for such logics in the setting of dynamical systems, i.e.\ when a topological space is equipped with a continuous function.
Our goal is to prove the finite Kripke model property, completeness and decidability of logics with the Cantor derivative operator and the `next' operator $\bc$ over some prominent classes of dynamical systems: namely, those based on arbitrary spaces, on $T_D$ spaces (spaces validating the $4$ axiom $\square p \to \square\square p$) and on {\em scattered spaces} (see Section~\ref{secTangle} for definitions).
The reason for considering scattered spaces is to circumvent the lack of finite axiomatisability of $\bold{DTL}$ by restricting to a suitable subclass of all dynamical systems.
In the study of dynamical systems and topological modal logic, one often works with {\em dense-in-themselves} spaces.
This is a sensible consideration when modelling physical spaces, as Euclidean spaces are dense-in-themselves.
However, as we will see in Section~\ref{secTangle}, some technical issues that arise when studying $\bold{DTL}$ over the class of all spaces disappear when restricting our attention to scattered spaces, which in contrast have many isolated points.
Further, we consider dynamical systems where $f$ is a {\em homeomorphism,} i.e.\ where $f^{-1}$ is also a continuous function.
Such dynamical systems are called {\em invertible}. 

The basic dynamic $d$-logic we consider is $\bold{wK4C}$, which consists of $\bold{wK4}$ and the temporal axioms for the continuous function $f$.
In addition, we investigate two extensions of $\bold{wK4C}$:  $\bold{K4C}$ and $\bold{GLC}$.
As we will see, $\bold{K4C}$ is the $d$-logic of all $T_D$ dynamical systems, and $\bold{GLC}$ is the $d$-logic of all dynamical systems based on a scattered space.
Unlike the generic logic of the trimodal topo-temporal language $\mathcal{L}^{\circ*}_\lozenge$, we conjecture that a complete finite axiomatisation for $\bold{GLC}$, extended with axioms for the `henceforth' operator, will not require changes to the trimodal language.
This logic is of special interest to us as it would allow for the first finite axiomatisation and completeness results for a logic based on the trimodal topo-temporal language.

\subparagraph*{Outline.} This paper is structured as follows: in Section 2 we give the required definitions and notations necessary to understand the paper. 
In Section 3 we provide some background on prior work on the topic of dynamic topological logics. Moreover, we motivate our interest in $\bold{GLC}$, the most unusual logic we work with.

In Section 4 we present the canonical model, and in Section 5 we construct a `finitary' accessibility relation on it.
Both are then used in Section 6 in order to develop a proof technique that, given the right modifications, would work for many $d$-logics above $\bold{wK4C}$. 
In particular, we use it to prove the finite model property, soundness and completeness for the $d$-logics $\bold{wK4C}$, $\bold{K4C}$ and $\bold{GLC}$, with respect to the appropriate classes of Kripke models.

In Section 7 we prove topological $d$-completeness of $\bold{K4C}$, $\bold{wK4C}$ and $\bold{GLC}$ with respect to the appropriate classes of dynamical systems.
In Section 8 we present logics for systems with homeomorphisms and provide a general completeness result which, in particular, applies to the $d$-logics $\bold{wK4H}$, $\bold{K4H}$ and $\bold{GLH}$.
Finally, in Section 9 we provide some concluding remarks.

\section{Preliminaries}\label{secPrel}
In this section we review some basic notions required for understanding this paper.
We work with the general setting of {\em derivative spaces}, in order to unify the topological and Kripke semantics of our logics.

\begin{defi} [topological space]
A topological space is a pair $\mathfrak {X}=\langle X,\tau\rangle$, where $X$ is a set and $\tau$ is a subset of $\wp(X)$ that satisfies the following conditions:
\begin{itemize}
\item $X,\varnothing\in\tau$;
\item if $U,V\in\tau$, then $U\cap V\in\tau$;
\item if $\mathcal{U}\subseteq \tau$, then $\bigcup\mathcal{U}\in\tau$.
\end{itemize}

The elements of $\tau$ are called open sets, and we say that $\tau$ forms a topology on $X$.  A complement of an open set is called a \emph{closed} set. 
\end{defi}
\color{black}

The main operation on topological spaces we are interested in is the {\em Cantor derivative.}

\begin{defi}[Cantor derivative] Let $\mathfrak {X}=\langle X,\tau\rangle$ be a topological space. Given $S\subseteq X$, the \emph{Cantor derivative} of $S$ is the set $d(S)$ of all limit points of $S$, i.e.\ $x\in  d(S) \iff \forall U\in \tau $ s.t. $ x\in U,\; (U\cap S)\backslash\{x\}\neq\varnothing$.
We may write $d(S)$ or $dS$ indistinctly.
\end{defi}

Given subsets $A,B\subseteq X$, it is not difficult to verify that $d $ satisfies the following properties:
\begin{enumerate}
\item\label{itDOne} $d(\varnothing)=\varnothing$;
\item\label{itDTwo} $d(A\cup B)= d(A)\cup d(B)$;
\item\label{itDThree} $dd(A) \subseteq A\cup d(A)$.
\end{enumerate}
\noindent In fact, these conditions lead to a more general notion of {\em derivative spaces:}\footnote{Derivative spaces are a special case of {\em derivative algebras} introduced by Esakia \cite{EsakiaAlgebra}, where $\wp(X)$ is replaced by an arbitrary Boolean algebra.
}

\begin{defi}[derivative space]
A {\em derivative space} is a pair $\lb X,\rho\rb$, where $X$ is a set and $\rho\colon \wp(X) \to \wp(X)$ is a map satisfying properties~\ref{itDOne}-\ref{itDThree} (with $\rho$ in place of $d$).
\end{defi} 

When working with more than one topological space, we will often denote the map $\rho$ on the topological space $\langle X,\tau\rangle$ by $\rho_\tau $. The intended derivative spaces discussed in this paper are of the form $\lb X,d_\tau\rb$.
However, there are other examples of derivative spaces.
The standard topological closure may be defined by $c(A) = A\cup d(A)$.
Then, $\lb X,c\rb$ is also a derivative space, which satisfies the additional property $A\subseteq c(A)$ (and, {\em a fortiori,} $cc(A) = c(A)$), which together form Kuratowski's axioms; we call such derivative spaces {\em closure spaces.}
Similarly to the Cantor derivative, we will often denote the closure of the topological space $\langle X,\tau\rangle$ by $c_\tau$.
Then, a set $A$ is closed iff $c_\tau(A) \subseteq A$~\cite{Munkres}.
%If $d=d_\tau$ for some topology $\tau$, then we denote $c$ by $c_\tau$.

Note that if $\rho=d_\tau$, then the topology $\tau$ is uniquely determined, but not ever derivative operator is of the form $d_\tau$.
In particular, weakly transitive Kripke frames provide examples of derivative spaces which do not necessarily arise from a topology.
For the sake of succinctness, we call these frames {\em derivative frames.}

\begin{defi}[derivative frame]
A {\em derivative frame} is a pair $\mathfrak  F = \langle W,\sqsubset \rangle$ where $W$ is a non-empty set and $\sqsubset $ is a {\em weakly transitive} relation on $W$, meaning that $w\sqsubset  v\sqsubset  u$ implies that $w\sqsubseteq u$, where $\sqsubseteq$ is the reflexive closure of $\sqsubset $.
\end{defi}
Below and throughout the text, we write $\exists x\sqsupset y \ \varphi $ instead of $\exists x (y\sqsubset x \wedge \varphi) $, and adopt a similar convention for the universal quantifier and other relational symbols.
We chose the notation $\sqsubset $ because it is suggestive of a transitive relation, but remains ambiguous regarding reflexivity, as there may be irreflexive and reflexive points.
Note that $\sqsubset$ is weakly transitive iff $\sqsubseteq$ is transitive. Given $A\subseteq W$, we define ${\downarrow_\sqsubset}$ as a map
${\downarrow_\sqsubset} \colon \wp(W) \to \wp(W)$ such that
$${\downarrow_\sqsubset}(A)= \{w\in W: \exists v\sqsupset w (v\in A)\}.$$
Similarly, we define
$${\uparrow_\sqsubset}(A)= \{w\in W: \exists v\sqsubset w (v\in A)\}.$$
The following is then readily verified:

\begin{lem}
If $\langle W,\sqsubset \rangle$ is a derivative frame, then $\langle W,{\downarrow_\sqsubset} \rangle$ is a derivative space.
\end{lem}

There is a connection between derivative frames and topological spaces.
Given a derivative frame $\langle W,\sqsubset \rangle$, we define a topology $\tau_\sqsubset $ on $W$ such that $U\in \tau_\sqsubset $ iff $U$ is upwards closed under $\sqsubset $, in the sense that if $w\in U$ and $v \sqsupset w$ then $v\in U$.
Topologies of this form are {\em Aleksandrov topologies.}
The following is well-known and easily verified.

\begin{lem}\label{lemmDtau}
Let $\langle W,\sqsubset \rangle$ be a derivative frame and $\tau=\tau_\sqsubset $.
Then, $d_\tau={\downarrow_\sqsubset} $ iff $\sqsubset  $ is irreflexive and $c_\tau={\downarrow_\sqsubset} $ iff $\sqsubset  $ is reflexive.
\end{lem}

Dynamical systems consist of a topological space equipped with a continuous function.
Recall that if $\langle X,\tau\rangle$ and $\langle Y,\upsilon\rangle$ are topological spaces and $f\colon X\to Y$, then $f$ is {\em continuous} if whenever $U\in \upsilon$, it follows that $f^{-1}(U)\in \tau$ or, equivalently, if whenever $A$ is closed in $Y$, then $f^{-1}(A)$ is closed in $X$.
The function $f$ is {\em open} if $f(V)$ is open whenever $V$ is open, and $f$ is a {\em homeomorphism} if $f$ is continuous, open and bijective.
It is well-known (and not hard to check; see e.g.~\cite{Munkres}) that $f$ is continuous iff $c_\tau f^{-1}(A) \subseteq f^{-1}c_\upsilon(A)$ for all $A\subseteq Y$.
By unfolding the definition of the closure operator, this becomes $ f^{-1}(A)\cup d_\tau f^{-1}(A) \subseteq  f^{-1} (A)\cup  f^{-1}d_\upsilon(A)$, or equivalently, $  d_\tau f^{-1}(A) \subseteq  f^{-1} (A)\cup  f^{-1}d_\upsilon(A)$.
We thus arrive at the following definition.

\begin{defi}\label{defCH}
Let $\langle X,\rho_X \rangle$ and $\langle Y,\rho_Y\rangle $ be derivative spaces.
We say that $f\colon X\to Y$ is {\em continuous} if for all $A\subseteq Y$, $  \rho_X f^{-1}(A) \subseteq  f^{-1} (A)\cup  f^{-1} \rho_Y (A)$.
We say that $f$ is a {\em homeomorphism} if it is bijective and $  \rho_X  f^{-1}(A) =  f^{-1} \rho_Y (A)$.
\end{defi}

It is worth checking that these definitions coincide with their standard topological counterparts.

\begin{lem}
If $\langle X,\tau\rangle $ and $\langle Y,\upsilon \rangle $ are topological spaces with Cantor derivatives $d_\tau$ and $d_\upsilon$ respectively, and $f\colon X\to Y$, then
\begin{enumerate}

\item  $f$ is continuous as a function between topological spaces if and only if it is continuous as a function between derivative spaces, and

\item $f$ is a homeomorphism as a function between topological spaces if and only if it is a homeomorphism as a function between derivative spaces.

\end{enumerate}
\end{lem}

\begin{proof}
We prove the first claim and leave the second to the reader.
Suppose that $f\colon X\to Y$ is continuous in the topological sense and let $x\in  \rho_X f^{-1}(A) $.
If $f(x)\in A$ then $x\in f^{-1}(A)\cup f^{-1} \rho_Y (A)$, so we may assume otherwise.
Let $U$ be any neighbourhood of $f(x)$; note that $f^{-1}(U)$ is open, since $f$ is continuous.
Since $x\in  \rho_X f^{-1}(A)$, there is $y\in f^{-1}(U) \setminus \{x\}$ such that $f(y) \in A$.
Since by assumption $f(x)\notin A$, we obtain $y\neq f(x)$, and since $U$ was arbitrary, $f(x) \in \rho_Y (A)$, so that $x\in f^{-1}(U) \cup  f^{-1} \rho_Y (A)$, as needed.

Now suppose that $f\colon X\to Y$ is continuous as a function between derivative spaces and let $A\subseteq Y$ be closed.
Then, $c_\tau f^{-1}(A) = f^{-1}(A) \cup \rho_X f^{-1}(A) \subseteq f^{-1}(A) \cup  f^{-1} \rho_Y (A) =f^{-1} (A\cup \rho_Y (A)) = f^{-1} (c_\upsilon(A)) = f^{-1}(A) $, where the last equality uses that $A$ is closed.
But then, $c_\tau f^{-1}(A) \subseteq f^{-1}(A) $ and hence $f^{-1}(A)$ is closed, as needed.
\end{proof}

We are particularly interested in the case where $X=Y$, which leads to the notion of {\em dynamic derivative system.}

\begin{defi}
A {\em dynamic derivative system} is a triple $\mathfrak S =\lb X,\rho,f\rb$, where $\lb X,\rho\rb$ is a derivative space and $f\colon X\to X$ is continuous.
If $f$ is a homeomorphism, we say that $\mathfrak S$ is {\em invertible.}
\end{defi}

If $\mathfrak S = \langle X,\rho,f\rangle$ is such that $\rho = d_\tau$ for some topology $\tau$, we say that $\mathfrak S$ is a {\em dynamic topological system} and identify it with the triple $\langle X,\tau,f \rangle $.
If $\rho={\downarrow_\sqsubset }$ for some weakly transitive relation $\sqsubset$, we say that $\mathfrak S$ is a {\em dynamic Kripke frame} and identify it with the triple $\langle X,\sqsubset ,f\rangle$.

It will be convenient to characterise dynamic Kripke frames in terms of the relation $\sqsubset $.

\begin{defi}[monotonicity and weak monotonicity]
Let $\langle W,\sqsubset \rangle$ be a derivative frame.
A function $f\colon W\to W$ is {\em monotonic} if $w\sqsubset  v$ implies $f(w)\sqsubset  f(v)$, and {\em weakly monotonic} if $w\sqsubset  v$ implies $f(w)\sqsubseteq f(v)$.

The function $f$ is {\em persistent} if it is a bijection and for all $w,v\in W$, $w\sqsubset  v$ if and only if $f(w) \sqsubset  f(v)$.
We say that a Kripke frame is {\em invertible} if it is equipped with a persistent function.
\end{defi}

\begin{lem}
If $\langle W,\sqsubset \rangle$ is a derivative frame and $f\colon W\to W$, then
\begin{enumerate}

\item if $f$ is weakly monotonic if and only if it is continuous with respect to $\downarrow_\sqsubset $, and

\item if $f$ is persistent if and only if it is a homeomorphism with respect to $\downarrow_\sqsubset $.

\end{enumerate}
\end{lem}

Our goal is to reason about various classes of dynamic derivative systems	using the logical framework defined in the next section.

\section{Dynamic Topological Logics}\label{secDTL}

In this section we discuss dynamic topological logic in the general setting of dynamic derivative systems.
Given a non-empty set $\mathsf{PV}$ of propositional variables, the language $\mathcal{L}_{\lozenge}^{\circ}$ is defined recursively as follows:
$$
\varphi::= p \; | \; \varphi\wedge \varphi \; | \; \neg\varphi \; | \; \lozenge\varphi  \; | \; \bc\varphi,$$
%$$p \in \mathsf{PV}.$$
where $p \in \mathsf{PV}.$
It consists of the Boolean connectives $\wedge$ and $\neg$, the temporal modality $\bc$, and the modality $\lozenge$ for the derivative operator with its dual $\square := \neg\lozenge\neg$.
The interior modality may be defined as $\boxdot\varphi:=\varphi\wedge\square\varphi$.

\begin{defi}[semantics]\label{d-semantics}
A \emph{dynamic derivative model} (DDM) is a quadruple $\mathfrak {M}=\langle X,\rho,f,\nu\rb$ where $\langle X,\rho,f\rb$ is a dynamic derivative system and $\nu:\mathsf{PV}\rightarrow\wp(X)$ is a valuation function assigning a subset of $X$ to each propositional letter in $\mathsf{PV}$.  
Given $\varphi \in\mathcal{L}_{\lozenge}^{\circ }$, we define the truth set $\|\varphi\| \subseteq X$ of $\varphi$ inductively as follows:
\begin{multicols}2
\begin{itemize}

	\item $\| p \| = \nu(p)$;
	\item $\| \neg \varphi \| = X \setminus \|  \varphi \| $;
	\item $\| \varphi\wedge\psi \| = \|\varphi\| \cap \|\psi\| $;
	\item $\|\lozenge\varphi\| = \rho ( \|\varphi\| ) $;
		\item $\| \bc\varphi \| = f^{-1}(\| \varphi\|)$.
		\item[]
%		\item $x\models_d[*]\varphi\iff y\models\varphi$ for $y=f^n(x)$ for all $n\ge 1$;
%		\item $x\models_d\forall\varphi\iff y\models\varphi$ for all $y\in X$.

\end{itemize}
\end{multicols}
We write $\mathfrak  M,x\models\varphi$ if $x\in \|\varphi\|$, and $\mathfrak  M\models\varphi$ if $\|\varphi\| = X$.
We may write $\|\cdot\|_\mathfrak M$ or $\|\cdot\|_\nu$ instead of $\|\cdot\|$ when working with more than one model or valuation.
\end{defi}

We define other connectives (e.g.\ $\vee,\rightarrow$) as abbreviations in the usual way.
The fragment of $\mathcal{L}_{\lozenge}^{\circ }$ that includes only $\lozenge$ will be denoted by $\mathcal L_{\lozenge}$.
Since our definition of the semantics applies to any derivative space and a general operator $\rho$, we need not differentiate in our results between $d$-logics, logics based on closure semantics and logics based on relational semantics. Instead, we indicate the specific class of derivative spaces to which the result applies.

In order to keep with the familiar axioms of modal logic, it is convenient to discuss the semantics of $\dn$.
Accordingly, we define the dual of the derivative, called the \emph{co-derivative}.

% one of the two modalities will be denoted by omitting the second, i.e.~$\mathcal{L}_{\lozenge} $ is the unimodal language with $\dd$ (and $\dn$).
\begin{defi}[co-derivative]
Let $\langle X,\rho\rangle$ be a derivative space.
For each $S\subseteq X$ we define $\hat \rho(S):=X\backslash  \rho(X\backslash S)$ to be the {\em co-derivative} of $S$.
\end{defi}

The co-derivative satisfies the following properties, where $A,B\subseteq X$:
\begin{enumerate}
\item $\hat \rho (X)=X$;
\item $A\cap \hat \rho(A)\subseteq \hat \rho \hat \rho(A)$;
\item $\hat \rho(A\cap B)=\hat \rho(A)\cap \hat \rho(B)$.
\end{enumerate}

It can readily be checked that for each dynamic derivative model $\langle X,\rho,f,\nu\rangle$ and all formulas $\varphi$, $\|\dn\varphi\| = \hat \rho (\|\varphi\|)$.
The co-derivative can be used to define the standard {\em interior} of a set, given by $i(A)=A\cap\hat \rho(A)$ for each $A\subseteq X$. 
This implies that $U\subseteq \hat \rho(U)$ for each open set $U$, but not necessarily $\hat \rho(U)\subseteq U$.
Next, we discuss the systems of axioms that are of interest to us.
Let us list the axiom schemes and rules that we will consider in this paper:
\begin{multicols}2
\begin{description}
\item{\rm Taut} $:= \text{All propositional tautologies}$
\item{\rm K} $:= \dn(\varphi\to\psi)\to(\dn\varphi\to \dn\psi)$
\item{\rm T} $:=    \dn\varphi \to \varphi $
\item{\rm w4} $ :=  \varphi \wedge  \dn\varphi \to\dn\dn\varphi $
\item{\rm L} $:=    \square(\square \varphi \rightarrow \varphi) \rightarrow \square \varphi  $
\item{\rm 4} $  :=    \dn\varphi \to\dn\dn\varphi $
\item{${\rm Next}_\neg$} $:=\neg\bc\varphi\leftrightarrow\bc\neg\varphi$
\item{${\rm Next}_\wedge$} $:=\bc (\varphi\wedge\psi)\leftrightarrow \bc \varphi \wedge\bc \psi $
\item{\rm C} $:= \bc\varphi\wedge \bc\dn\varphi\to  \dn\bc\varphi$
\item{\rm H} $ := \dn\bc\varphi\leftrightarrow\bc\dn\varphi$
\item{\rm MP} $:= \dfrac{\varphi \ \ \varphi\to \psi}\psi$
\item{${\rm Nec}_\dn$} $:= \dfrac{\varphi }{\dn \varphi}$
\item{${\rm Nec}_\bc$} $:= \dfrac{\varphi }{\bc \varphi}$
\end{description}
\end{multicols}
The `base modal logic' over $\mathcal L_\dd$ is given by
$$\mathbf{K}:= {\rm Taut}+{\rm K} +{\rm MP}+{\rm Nec}_\dn,$$
but we are mostly interested in proper extensions of $\mathbf K$.
Let $\Lambda$ and $\Lambda'$ be logics over languages $\mathcal L$ and ${\mathcal L}'$ respectively. 
We say that $\Lambda$ extends $\Lambda'$ if $\mathcal L'\subseteq\mathcal L$ and all the axioms and rules of $\Lambda'$ are derivable in $\Lambda$.
A logic over $\mathcal L_\dd$ is {\em normal} if it extends $\mathbf{K}$.
If $\Lambda$ is a logic and $\varphi$ is a formula, $\Lambda+\varphi$ is the least extension of $\Lambda$ which contains every substitution instance of $\varphi$ as an axiom.

We then define $\mathbf{wK4}:= \mathbf{K}+{\rm w4}$, $\mathbf{K4}:= \mathbf{K}+{\rm 4}$,
$\mathbf{S4}:=  \mathbf{K4}+{\rm T}$ and $\mathbf{GL}:=  \mathbf{K4}+{\rm L}$.
These logics are well-known and characterise certain classes of topological spaces and Kripke frames which we review below.
In addition, for a logic $\Lambda$ over $\mathcal L_\dd$, $\Lambda\mathbf{F}$ is the logic over $\mathcal L^\circ_\dd$ given by\footnote{Logics of the form $\Lambda\mathbf{ F}$ correspond to dynamical systems with a possibly discontinuous function. We will not discuss discontinuous systems in this paper; see \cite{artemov} for more information.}
$$\Lambda\mathbf{F} := \Lambda+{\rm Next}_\neg+{\rm Next}_\wedge+{\rm Nec}_\bc.$$
This simply adds axioms of linear temporal logic to $\Lambda$, which hold whenever $\bc$ is interpreted using a function.
Finally, we define $\Lambda{\bf C}:=\Lambda\mathbf{F}+\rm C$ and $\Lambda{\bf H}:=\Lambda\mathbf{F}+\rm H$, which as we will see correspond to derivative spaces with a continuous function or a homeomorphism respectively.
The purely topological fragments of these logics have been well-studied, starting with the following well-known result dating back to McKinsey and Tarski \cite{Tarski}.

\begin{thm}
$\mathbf{S4}$ is the logic of all topological closure spaces, the logic of all transitive, reflexive derivative frames, and the logic of the real line with the standard closure.
\end{thm}

Analogously, Esakia demonstrated that $\bold{wK4}$ is the logic of topological derivative spaces \cite{Esakia2}.

\begin{thm}
$\bold{wK4}$ is the logic of all topological derivative spaces, as well as the logic of all weakly-transitive derivative frames.
\end{thm}

The logic $\bold{K4}$ includes the axiom $\square p \to \square\square p$, which is not valid over the class of all topological spaces.
The class of spaces satisfying this axiom is denoted by $T_D$, defined as the class of spaces in which every singleton is the result of an intersection between an open set and a closed set.
Moreover, Esakia showed that this is the logic of transitive derivative frames \cite{EsakiaAlgebra}.

\begin{thm}
$\mathbf{K4}$ is the logic of all $T_D$ topological derivative spaces, as well as the logic of all transitive derivative frames.
\end{thm}

Many familiar topological spaces, including Euclidean spaces, satisfy the $T_D$ property, making $\mathbf{K4}$ central in the study of topological modal logic.
A somewhat more unusual class of spaces, which is nevertheless of particular interest to us, is the class of {\em scattered spaces.}

\begin{defi}[scattered space]
A topological space $\lb X,\tau\rb$ is {\em scattered} if for every $S\subseteq X$,
$ S\subseteq d(S) \text{ implies } S=\varnothing.$
\end{defi}
This is equivalent to the more common definition of a scattered space where a topological space is called scattered if every non-empty subset has an isolated point.
Scattered spaces are closely related to converse well-founded relations.
Below, recall that $\langle W,\sqsubset \rangle$ is {\em converse well-founded} if there is no infinite sequence $w_0\sqsubset  w_1\sqsubset  \ldots$ of elements in $W$.

\begin{lem}\label{lemScatteredKrip}
If $\langle W,\sqsubset \rangle$ is an irreflexive frame, then $\langle W,\tau_\sqsubset \rangle$ is scattered iff $\sqsubset $ is converse well-founded.
\end{lem}

\begin{proof}
First suppose that $\langle W,\tau_\sqsubset \rangle$ is scattered and let $A\subseteq W$ be non-empty.
Let $x$ be an isolated point of $A$.
Then, $x$ is isolated in $A$, and $U:=\{w\}\cup {\uparrow_\sqsubset}(\{w\})$ is the least neighbourhood of $w$ so we must have that $U\cap A = \{w\}$, i.e.~$w$ is a $\sqsubset$-maximal point of $A$.

Conversely, if $\langle W, \sqsubset \rangle$ is converse well-founded and $A\subseteq W$ is non-empty, it is readily checked that any $\sqsubset$-maximal point of $A$ is isolated in $A$.
\end{proof}

\begin{thm}[Simmons \cite{simmons} and Esakia \cite{esakia}]\label{thmGLComp}
$\mathbf{GL}$ is the logic of all scattered topological derivative spaces, as well as the logic of all converse well-founded derivative frames and the logic of all finite, transitive, irreflexive derivative frames.
\end{thm}

Aside from its topological interpretation, the logic $\mathbf{GL}$ is of particular interest as it is also the logic of provability in Peano arithmetic, as was shown by Boolos \cite{boolo}.
Meanwhile, logics with the $\rm C$ and $\rm H$ axioms correspond to classes of dynamical systems.

\begin{lem}\label{lemCH}\
\begin{enumerate}

\item If $\Lambda$ is sound for a class of derivative spaces $\Omega$, then $\Lambda{\bf C}$ is sound for the class of dynamic derivative systems $\langle X,\rho,f\rangle$, where $\langle X,\rho\rangle\in \Omega$ and $f$ is continuous.

\item If $\Lambda$ is sound for a class of derivative spaces $\Omega$, then $\Lambda{\bf H}$ is sound for the class of dynamic derivative systems $\langle X,\rho,f\rangle$, where $\langle X,\rho\rangle\in \Omega$ and $f$ is a homeomorphism.

\end{enumerate}
\end{lem}

The above lemma is easy to verify from the definitions of continuous functions and homeomorphisms in the context of derivative spaces (Definition~\ref{defCH}).

\subsection{Prior Work}

The study of dynamic topological logic originates with Artemov, Davoren and Nerode, who observed that it is possible to reason about dynamical systems within modal logic.
They introduced the logic $\bf S4C$ and proved that it is decidable, as well as sound and complete for the class of all dynamic closure systems (i.e.\ dynamic derivative systems based on a closure space).
Kremer and Mints \cite{kremer} considered the logic $\bf S4H$ and showed that it is sound and complete for the class of dynamic closure systems where $f$ is a homeomorphism.

The latter also suggested adding the `henceforth' operator, $*$, from Pnueli's linear temporal logic ($\bf LTL$) \cite{ltl}, leading to the language we denote by $\mathcal{L}^{\circ*}_\dd$.
The resulting trimodal system was named {\em dynamic topological logic} ($\bf DTL$).
Kremer and Mints offered an axiomatisation for $\bf DTL$, but Fern\'andez-Duque proved that it is incomplete; in fact, $\bf DTL$ is not finitely axiomatisable \cite{david}.
Fern\'andez-Duque also showed that $\bf DTL$ has a natural axiomatisation when extended with the {\em tangled closure} \cite{david3}.
In contrast, Konev et al.\ established that $\bf DTL$ over the class of dynamical systems with a homeomorphism is non-axiomatisable \cite{KonevMetric}.

\subsection{The tangled closure on scattered spaces}\label{secTangle}

Our interest in considering the class of scattered spaces within dynamic topological logic is motivated by results of Fern\'andez-Duque \cite{david3}.
He showed that the set of valid formulas of $\mathcal{L}^{\circ*}_\dd$ over the class of all dynamic closure systems is not finitely axiomatisable.
Nevertheless, he found a natural (yet infinite) axiomatisation by introducing the \emph{tangled closure} and adding it to the language of $\bf DTL$ \cite{FernandezTangled}.
Here, we use the more general {\em tangled derivative}, as defined by Goldblatt and Hodkinson \cite{Goldblatt2017Spatial}.

\begin{defi}[tangled derivative]
Let $\lb X,\rho\rb$ be a derivative space and let $\mathcal{S}\subseteq \wp(X)$. Given $A\subseteq X$, we say that $\mathcal{S}$ is tangled in $A$ if for all $S\in \mathcal{S}$,
$ A\subseteq \rho(S\cap A) $.
We define the \emph{tangled derivative} of $\mathcal{S}$ as 
$$ \mathcal{S}^*:=\bigcup \{ A\subseteq X : \mathcal{S}\text{ is tangled in }A\}.$$ 
\end{defi}

The {\em tangled closure} is then the special case of the tangled derivative where $\rho$ is a closure operator, i.e.~$\rho=c$.
Fern\'andez-Duque's axiomatisation is based on the extended language $\mathcal L^{\circ*}_{\scriptsize\loz}$.
This language is obtained by extending $\mathcal L^{\circ*}_{\dd}$ with the following operation. 
\begin{defi}[tangled language]
We define $\mathcal L^{\circ*}_{\scriptsize\loz}$ by extending the recursive definition of $\mathcal L^{\circ*}_{\dd}$ in such a way that if $\varphi_1,\ldots,\varphi_n\in \mathcal L^{\circ*}_{\scriptsize\loz}$, then $\loz\{\varphi_1,\ldots,\varphi_n\} \in \mathcal L^{\circ*}_{\scriptsize\loz}$.
The semantic clauses are then extended so that on every model $\mathfrak M$,
$$\| {\loz\{\varphi_1,\ldots,\varphi_n\}}\| = \{\|\varphi_1\|,\ldots,\|\varphi_n\|\}^* .$$
\end{defi}

 The logic $\bold{DGL}$ is an extension of $\bold{GLC}$ that includes the temporal operator $*$. Unlike the complete axiomatisation of $\mathbf{DTL}$ that requires the tangled operator, in the case of $\bold{DGL}$, we should be able to avoid this and use the original spatial operator $\lozenge$ alone. This is due to the following:

\begin{restatable}{thm}{scatteredlozenge}
Let $\mathfrak{X}=\lb X,\tau \rb$ be a scattered space and $\{\varphi_1,\dots,\varphi_n\}$ a set of formulas. Then
$$ \loz\{\varphi_1,\dots,\varphi_n\}\equiv \bot.$$
\end{restatable}
\begin{proof}
Let $\nu$ be a valuation on $\mathfrak X$, $\mathfrak M=\langle \mathfrak X,\nu\rangle$, and suppose that $\mathfrak M,x\models\loz\{\varphi_1,\dots ,\varphi_n\}$ for some $x\in X$.
Then by definition there exists $S\subseteq X$ s.t. $x\in S$ and $ S\subseteq d( S\cap \|\varphi_i\|)$, for all $i\leq n$. It follows that $S\subseteq d(S)$. Since $\mathfrak{X}$ is scattered, then $S=\varnothing$ in contradiction. It follows that $\mathfrak M,x\not\models\loz\{\varphi_1,\dots ,\varphi_n\}$ for all $x\in X$.
\end{proof}

This leads to the conjecture that the axiomatic system of Kremer and Mints \cite{kremer}, combined with $\mathbf{GL}$, will lead to a finite axiomatisation for $\mathbf{DGL}$.
While such a result requires techniques beyond the scope of the present work, the completeness proof we present here for $\mathbf{GLC}$ is an important first step.
Before proving topological completeness for this and the other logics we have mentioned, we show that they are complete  and have the finite model property for their respective classes of dynamic derivative frames.

\section{The Canonical Model}

The first step in our Kripke completeness proof will be a fairly standard canonical model construction.
For a given logic $\Lambda$, a \emph{maximal $\Lambda$-consistent set} ($\Lambda$-MCS) $w$ is a set of formulas that is $\Lambda$-consistent, i.e.\ $w\nvdash_{\Lambda} \bot$, and every set of formulas that properly contains it is $\Lambda$-inconsistent.

Given a logic $\Lambda$ over $\mathcal L^{\circ}_\dd$, let $\mathfrak{M}^\Lambda_{\rm c}=\langle W_{\rm c}, \sqsubset _{\rm c} ,g_{\rm c},\nu_{\rm c}\rangle$ be the canonical model for $\Lambda$, where
 \begin{enumerate}
	\item $W_{\rm c}$ is the set of all $\Lambda$-MCSs;
	\item $w\sqsubset _{\rm c}v$ iff for all formulas $\varphi$, if $\square\varphi\in w$, then $\varphi\in v$;

\item $g_{\rm c}(w)=\{\varphi:\bc\varphi\in w\}$;
\item $\nu_{\rm c}(p)=\{w:p\in w\}$.
\end{enumerate}

It can easily be verified that $\bold{wK4C}$ defines the class of all weakly transitive Kripke models with a weakly monotonic map. Moreover, $\bold{K4C}$ defines the class of all transitive Kripke models with a weakly monotonic map. We call these models $\bold{wK4C}$ models and $\bold{K4C}$ models respectively.

\begin{restatable}{lem}{canonical}
If $\Lambda$ extends $\bold{wK4C}$, then the canonical model for $\Lambda$ is a $\bold{wK4C}$ model. If $\Lambda$ extends $\bold{K4C}$, then the canonical model of $\Lambda$ is a $\bold{K4C}$ model.
\end{restatable}
\begin{proof}
Let $\mathfrak{M}^\Lambda_{\rm c}=\langle W , \sqsubset   ,g ,\nu \rangle$.
Suppose that $\Lambda$ extends $\bold{wK4C}$. We prove that $g $ is weakly monotonic. Suppose $w\sqsubset  v$ and $g (w)\neq g (v)$. Since $g (w)\neq g (v)$, there exists $\varphi$ such that $\varphi\in g (w)$ and  $\varphi\notin g (v)$. We consider an arbitrary $\square\psi\in g (w)$ then clearly $\square (\psi\vee\varphi)\in g(w)$ and $(\psi\vee\varphi)\in g(w)$. In particular, $\bc\square (\psi\vee\varphi)\in w$ and $\bc(\psi\vee \varphi)\in w$. Since $(\bc\square p\wedge \bc p)\rightarrow \square\bc p\in \bold{wK4C}$, it follows that $\square\bc(\varphi\vee\psi)\in w$. From  $w\sqsubset  v$, we obtain $\bc(\varphi\vee\psi)\in v$ and $(\varphi\vee\psi)\in g(v)$. Since $\varphi\notin g(v)$, $\psi\in g(v)$ and because $\square\psi$ is arbitrary, $g(w) \sqsubset   g(v)$, as required. Hence $g$ is weakly monotonic.

 We prove that $\sqsubset $ is weakly transitive. Suppose $w\sqsubset  v\sqsubset  u$ and $w\neq u$. From $w\neq u$ it follows that there exists $\varphi$ such that $\varphi\in w$ and $\varphi\notin u$. We consider an arbitrary $\square\psi\in w$.
 Clearly $\square(\psi\vee\varphi)\in w$ and $(\psi\vee\varphi)\in w$. Since $(\varphi\wedge\square \varphi)\rightarrow \square\square \varphi\in\bold{wK4C}$, we have that $\square\square(\psi\vee\varphi)\in w$. Then, $w\sqsubset  v\sqsubset  u$  implies $(\psi\vee\varphi)\in u$ and as $\varphi\notin u$, we have that $\psi\in u$. Since $\square\psi$ is arbitrary, $w\sqsubset   u$ holds, as required.
 
  It follows that $g$ is weakly monotonic and $\sqsubset  $ is weakly transitive and hence $\mathfrak{M}^\Lambda_{\rm c}$ is a $\bold{wK4C}$ model.
  
  Suppose that $\Lambda$ extends $\bold{K4C}$, then weak monotonicity holds as before since $\bold{K4C}$ extends $\bold{wK4C}$. Therefore, we only need to prove that $\sqsubset  $ is transitive. Suppose $w\sqsubset  v\sqsubset   u$. We consider an arbitrary $\square\psi\in w$. Since $\square p\rightarrow \square\square p\in \bold{wK4C}$, then $\square\square \psi\in w$. Then, $w\sqsubset  v\sqsubset   u$ implies $\psi\in u$ and since $\square\psi$ is arbitrary,   $w \sqsubset   u$, as required.
  
  It follows that $g$ is weakly monotonic and $\sqsubset  $ is transitive and thus $\mathfrak{M}^\Lambda_{\rm c}$ is a $\bold{K4C}$ model.
\end{proof}

It is a well-known fact that the transitivity axiom $\square p\rightarrow \square\square p$ is derivable in $\bold{GL}$ (see \cite{trans}). Therefore, $\bold{GLC}$ extends the system $\bold{K4C}$. The proofs of the following two lemmas are standard and can be found for example in \cite{black}.

\begin{lem}[existence lemma]\label{lemExist}
Let $\Lambda$ be a normal modal logic and let $\mathfrak M^\Lambda_{\rm c} = \langle W_{\rm c},\sqsubset  _{\rm c},g_{\rm c},\nu_{\rm c}\rangle$.
Then, for every $w\in W_{\rm c}$ and every formula $\varphi$ in $\Lambda$, if $\lozenge\varphi\in w$ then there exists a point $v\in W_{\rm c}$ such that $w\sqsubset  _{\rm c} v$ and $\varphi\in v$.
\end{lem}

 \begin{lem}[truth lemma]
 Let $\Lambda$ be a normal modal logic.
For every $w\in W_{\rm c}$ and every formula $\varphi$ in $\Lambda$, $$\mathfrak M^\Lambda_{\rm c}, w\models \varphi \text{ iff }\varphi\in w.$$
 \end{lem}

\begin{cor}
The logic $\bold{wK4C}$ is sound and complete with respect to the class of all weakly monotonic dynamic derivative frames, and $\bold{K4C}$ is sound and complete with respect to the class of all weakly monotonic, transitive dynamic derivative frames.
\end{cor}

\section{A finitary accessibility relation}

One key ingredient in our finite model property proof will be the construction of a `finitary' accessibility relation $\sqsubset  _{\Phi}$ on the canonical model.
This accessibility relation will have the property that each point has finitely many successors, yet the existence lemma will hold for formulas in a prescribed finite set $\Phi$.
This is a kind of selective filtration (see \cite{AlexCha}).

We define {\em the $\sqsubset  _{\rm c}$-cluster} $C(w)$ for each point $w\in W_{\rm c}$ as 
$$ C(w)=\{w\}\cup\{v:w\sqsubset  _{\rm c}v\sqsubset  _{\rm c}w\}.$$

\begin{defi}[$\varphi$-final set]\label{finalset}
A set $w$ is said to be a $\varphi$-final set (or point) if $w$ is an MCS, $\varphi\in w$, and whenever $w \sqsubset  _{\rm c} v$ and $\varphi\in v$, it follows that $v\in C(w)$.
\end{defi}

Let us write $\sqsubseteq_{\rm c}$ for the reflexive closure of $\sqsubset_{\rm c}$.
It will be convenient to characterise $\sqsubseteq_{\rm c}$ in the canonical model syntactically.
Recall that $\boxdot\varphi:= \varphi\wedge\square \varphi$.

\begin{restatable}{lem}{extend}\label{lemmBoxDot}
If $\Lambda$ extends $\bold{wK4C}$ and $w,v\in W_{\rm c}$, then $w\sqsubseteq_{\rm c} v$ if and only if whenever $\boxdot\varphi\in w $, it follows that $\boxdot\varphi\in v$.
\end{restatable}
\begin{proof}
First assume that $w\sqsubseteq_{\rm c} v$.
Obviously if $w=v$ then $\boxdot\varphi \in w$ implies $\boxdot\varphi \in v$.
If $\boxdot\varphi \in w$ then from $\square\varphi\in w$ we obtain $\varphi\in v$ and by $\rm w4$ we have $\square\square\varphi\in w$ and so $\square\varphi \in v$.

Conversely, assume that $\boxdot\varphi \in w$ implies $\boxdot\varphi \in v$.
If $w = v$, then $w\sqsubseteq_{\rm c} v$.
Otherwise, $w\neq v$, and we can find $\psi \in w\setminus v$.
Suppose that $\square\varphi \in w$.
Then $\boxdot(\varphi\vee\psi )\in w$ and so $\boxdot(\varphi\vee\psi )\in v$. It follows that $\varphi\vee\psi\in v$.
But $\psi\not\in v $, so $\varphi\in v$; since $\varphi$ was arbitrary, then $w\sqsubset_{\rm c} v$, as required.
\end{proof}

The following version of Zorn’s lemma will be used to prove an important existence property.

\begin{lem}[Zorn's Lemma]\label{choice}
Let $(A,\leq)$ be a preordered set where $A$ is non-empty. A chain is a subset $\mathcal{C} \subseteq A$ whose elements are totally ordered by $\leq$. Suppose that every chain $\mathcal{C}$ has an upper bound in $A$. Then, $A$ has a $\leq$-maximal element.
\end{lem}

The following proves a sufficient condition for the existence of a final point accessible from a given point.

\begin{lem}\label{choice2}
If $\lozenge\varphi\in w$, then there is $\varphi$-final point $v$ such that $w \sqsubset  _{\rm c} v$.
\end{lem}

\begin{proof}
Suppose that $\lozenge\varphi\in w$ and let $A:={\sqsubset  _{\rm c}}(w) \cap \|\varphi\|$ (where ${\sqsubset  _{\rm c}}(w)  = \{v\in W_{\rm c}: w\sqsubset  _{\rm c} v \} $).
Note that $A$ is non-empty by Lemma~\ref{lemExist}, so in order to apply Zorn's lemma, consider a $\sqsubseteq_{\rm c}$-chain $\mathcal{C}$ in $A$.
We show that there is an upper bound of $\mathcal{C}$ that belongs to $A$.

Choose a formula $\theta$ as follows:
If $w$ is an irreflexive world (i.e.\ $w\not\sqsubset_{\rm c} w$), then choose $\theta$ so that $\square\theta\wedge \neg\theta \in w$; such a $\theta$ exists since otherwise the definition of $\sqsubset_{\rm c}$ would yield $w\sqsubset_{\rm c} w$.
Otherwise, $w$ is reflexive and we simply set $\theta=\top$.
Let $\Gamma$ be the set 
$$ \{\boxdot\psi:\exists x\in \mathcal{C}(\boxdot\psi\in x)\}.$$

Suppose $\boxdot\psi_1,\dots,\boxdot\psi_n\in \Gamma$.
For each $i$ there is $w_i\in \mathcal C$ such that $\boxdot\psi_i \in w_i$.
Since $\mathcal C$ is a chain, for some $j$ we have that $w_j =\max_{i=1}^n w_i$.
Then, using Lemma~\ref{lemmBoxDot} we see that $\varphi,\theta, \boxdot\psi_1,\dots,\boxdot\psi_n\in w_j$, hence $\{\varphi,\theta,\boxdot\psi_1,\dots,\boxdot\psi_n\}$ is consistent.
Since $\{ \boxdot\psi_1,\dots,\boxdot\psi_n\}$ is an arbitrary finite subset of $\Gamma$, we have that $\Gamma \cup \{\varphi,\theta\}$ is consistent.

By the Lindenbaum lemma we can extend $\Gamma \cup \{\varphi,\theta\}$ to an MCS, which we denote by $w^*$. Suppose $\boxdot\psi \in v\in\mathcal{C}$, then by construction $\boxdot\psi \in w^*$.  It follows by Lemma~\ref{lemmBoxDot} that $v\sqsubseteq  _{\rm c}w^*$.
Hence $w^*$ is an upper bound for $\mathcal C$.
Note that the assumption that $\lozenge\varphi\in w$ implies that there is some $v\sqsupset_{\rm c} w$ such that $\varphi\in v$, that is, $v\in A$.
Hence $w\sqsubset_{\rm c} v\sqsubseteq_{\rm c} w^*$, which yields $w \sqsubseteq_{\rm c} w^*$.
To see that indeed $w \sqsubset_{\rm c} w^*$, if $w$ is reflexive there is nothing to prove, otherwise since $\theta\in w^*$ we have that $w\neq w^*$.
Moreover $\varphi\in w^*$ by construction, so $w^*\in A$.

Thus, by Lemma~\ref{choice} we conclude that $A$ contains a maximal element, and this maximal element is clearly a $\varphi$-final point above $w$.
\end{proof}

We are now ready to prove the main result of this section regarding the existence of the finitary relation $\sqsubset_\Phi$.
The idea is that $\sqsubset_\Phi$ is a sub-relation of $\sqsubset  _{\rm c}$ that only chooses enough points to provide witnesses for any formula $\lozenge\varphi\in  \Phi$.
We also want $w\sqsubset_\Phi v$ to depend only on the cluster of $w$, except possibly in the case that $w=v$; in other words, $\sqsubseteq_\Phi$ will be cluster-invariant.
This will allow us to pick finite submodels of the canonical model, by including only the $\sqsubset_\Phi$-successors of each world rather than all of the $\sqsubset  _{\rm c}$-successors (which are typically infinitely many).

\begin{restatable}{lem}{finitary}
Let $\Lambda$ extend $\bold{wK4C}$ and $\Phi$ be a finite set of formulas closed under subformulas. There is an auxiliary relation $\sqsubset  _{\Phi} $ on the canonical model of $\Lambda$ such that:
\begin{enumerate}[(i)]
\item $\sqsubset  _{\Phi}$ is a subset of $\sqsubset  _{\rm c}$;
\item For each $w\in W$, the set ${\sqsubset  _{\Phi}}(w)$ is finite;
\item If $\lozenge\varphi\in w\cap\Phi$, then there exists $v\in W$ with $w\sqsubset  _{\Phi}v$ and $\varphi\in v$;
\item If $w \sqsubset  _{\rm c} v \sqsubset  _{\rm c} w$ then ${\sqsubset  _{\Phi}(w)} \subseteq {\sqsubseteq  _{\Phi}(v)}  $;
\item $\sqsubset  _{\Phi}$ is weakly transitive.
Moreover, if $\Lambda$ extends $\bold{K4C}$ then $\sqsubset  _{\Phi}$ is transitive, and if $\Lambda$ extends $\bold{GLC}$ then $\sqsubset  _{\Phi}$ is irreflexive.
\end{enumerate}
\end{restatable}

\begin{proof} 
Let $C$ be any cluster of points in $W$ and define 
$$ { \sqsubset  _{\rm c}}(C)=\bigcup\{{\sqsubset  _{\rm c}}(v):v\in C\}.$$
 We construct the weakly transitive relation $\sqsubset  _{\Phi}$ as follows:
 Using Lemma~\ref{choice2} we use the axiom of choice to choose a function that for each formula $\varphi$ and each cluster $C$ such that $\lozenge\varphi \in \bigcup C$, assigns a $\varphi$-final point $w(\varphi,C)$ such that $w(\varphi,C)\in { \sqsubset  _{\rm c}}(C)$.
  We choose a second point $w'(\varphi,C)$, possibly equal to $w(\varphi,C)$, such that 
\begin{itemize}
 \item if $\lozenge\varphi\in w(\varphi,C)$, then $w'(\varphi,C)$ is any $\varphi$-final point such that $w(\varphi,C)\sqsubset  _{\rm c} w'(\varphi,C)$;
 \item otherwise, $w'(\varphi,C)=w(\varphi,C)$.
 \end{itemize}
Set $u \sqsubset  _{\Phi}^0 v$ iff $u\sqsubset_{\rm c} v$ and there exists $\psi\in\Phi$ such that $\lozenge\psi\in u$ and $$v\in\{w(\psi,C(u)),w'(\psi,C(u))\}.$$
 Let $\sqsubset  _{\Phi}$ be the weakly transitive closure of $\sqsubset  _{\Phi}^0$.

 It is clear that (i), (iii) and (iv) follow directly from the construction: (i) follows from the fact that $\sqsubset_\Phi$ is the transitive closure of $\sqsubset^0_\Phi$; (iii) follows from Lemma~\ref{choice2}; (iv) follows from the fact that $v\in C(w)$ and by assuming that $v\sqsubseteq_\Phi u$ and unraveling the definition of $\sqsubseteq_\Phi$.

We continue to verify conditions (ii) and (v). 
For condition (ii), first observe that ${\sqsubset_{\Phi}^0}(u)$ is finite by construction, as it contains at most two elements for each $\varphi\in \Phi$.
Now, if $u \sqsubset  _{\Phi}v$ then there is a sequence
$$ u\sqsubset  _{\Phi}^0v_1\sqsubset  _{\Phi}^0\dots \sqsubset  _{\Phi}^0v_n=v.$$
By taking a minimal such sequence, we may assume that it is injective, i.e.\ each element appears only once in the sequence.
Consider the tree consisting of all such sequences (ordered by the initial segment relation).
This is a finitely-branching tree, as ${\sqsubset_{\Phi}^0}(x)$ is always finite.
Moreover, if ${\sqsubset_{\Phi}}(u)$ is infinite, then this tree is infinite.
By K\"onig's lemma, there is an infinite sequence
$$ u \sqsubset  _{\Phi}^0v_1\sqsubset  _{\Phi}^0 v_2\sqsubset  _{\Phi}^0 \dots .$$

By definition of $\sqsubset_{\Phi}^0$, for each $i \in  \omega $ there is $\varphi_i \in \Phi$ such that $v_{i+1}$ is $\varphi_i$-final.
Since ${\sqsubset_{\Phi}^0} \subseteq {\sqsubset_{\rm c}}$, we have that $v_{i} \sqsubseteq_{\rm c} v_j$ whenever $i\leq j$.
Since $\Phi$ is finite, there is some $\theta\in\Phi$ such that $v_{i}$ is $\theta$-final for infinitely many values of $i$.
Let $i_0$ be the least such value.
If $i>i_0$ is any other such value, $v_{i_0} \sqsubset_{\rm c} v_i$ together with $\theta$-finality of $v_{i_0}$ yields $v_i \sqsubset_{\rm c} v_{i_0}$.
Thus $v_i\in C(v_{i_0})$ and $ v_{i} \sqsubset_{\Phi}^0 v_{i+1}$, which by definition of $\sqsubset^0_{\Phi}$ yields $  v_{i+1} \in {\sqsubset_{\Phi}^0} (v_{i_0} )$.
But ${\sqsubset_{\Phi}^0} (v_{i_0} )$ is finite, contradicting that the chain is infinite and injective.

Finally, we verify condition (v).
The relation $\sqsubset_\Phi$ is weakly transitive by definition, and if $\Lambda$ extends $\mathbf{K4C}$, we have that $x\sqsubset_\Phi y \sqsubset_\Phi z$ implies $x\sqsubseteq_\Phi z$, which in the case that $x=z$ implies that $x\sqsubset_{\rm c} x$ and also that $x,y$ are in the same cluster.
This together with $y \sqsubset_\Phi z $ implies that $x\sqsubset_\Phi z$, as $  \sqsubset_\Phi  $ is cluster-invariant.

Now let $\Lambda$ extend $\mathbf{GLC}$, and suppose that $w$ is $\varphi$-final.
We claim that $\lozenge \varphi \not \in w$, which immediately yields $w\not\sqsubset_\Phi w$, as needed. By the contrapositive of $\rm L$, we have that $\lozenge\varphi\to \lozenge(\varphi\wedge\square\neg\varphi ) \in w$.
If $\lozenge \varphi \in w$, then $\lozenge( \varphi\wedge\square\neg\varphi ) \in w$.
It follows that there is $v\sqsupset_\Phi w$ with $\varphi\wedge\square\neg\varphi \in v$.
But $\square\neg\varphi \in v$ implies that $v\not\sqsubset_\Phi w$, contradicting the $\varphi$-finality of $w$.
So $\lozenge \varphi \notin w$ and $w$ is irreflexive, as required.
\end{proof}

\section{Stories and $\Phi$-morphisms}
In this section we show that the logics $\bold{wK4C}$, $\bold{K4C}$ and $\bold{GLC}$ have the finite model property by constructing finite models and truth preserving maps from these models to the canonical model. 

If $\sqsubset$ is a weakly transitive relation on $A$, $\lb A,\sqsubset \rb$ is tree-like if whenever $a\sqsubseteq c$ and $b\sqsubseteq c$, it follows that $a\sqsubseteq b$ or $b\sqsubseteq a$.
We will use labelled tree-like structures called {\em moments} to record the `static' information at a point; that is, the structure involving $\sqsubset$, but not $f$.

\begin{defi}[moment]
A {\em $\Lambda$-moment} is a structure $\mathfrak m = \langle |\mathfrak m|,\sqsubset  _\mathfrak m,\nu_\mathfrak m,r_\moment \rangle $, where $\langle |\mathfrak m|,\sqsubset  _\mathfrak m \rangle$ is a finite tree-like $\Lambda$ frame with a root $r_\moment$, and $\nu_\mathfrak m$ is a valuation on $|\mathfrak m|$.\end{defi}

In order to also record `dynamic' information, i.e.~information involving the transition function, we will stack up several moments together to form a `story'.
Below, $ \bigsqcup$ denotes a disjoint union.

\begin{defi}[story]
A {\em story (with duration $I$)} is a structure  $\gog = \langle|\gog|,\sqsubset  _\gog,f_\gog,\nu_\gog,r_\gog \rangle$ such that there are $I< \omega$, $\Lambda$-moments $\gog_i = \langle |\gog_i|,\sqsubset  _i,\nu_i,r_i \rangle$ for each $i\leq I$, and functions $(f_i)_{i<I}$ such that:
\begin{enumerate}

\item $|\gog| = \bigsqcup_{i\leq I} |\gog_i| $;

\item $\sqsubset  _\gog = \bigsqcup_{i\leq I} \sqsubset  _i $;

\item $\nu_\gog(p) = \bigsqcup_{i\leq I} \nu_i(p) $ for each variable $p$;

\item $r_\gog =r_0$;

\item $f_\gog = {\rm Id}_{I} \cup \bigsqcup_{i <I} f_i $ with $f_i\colon |\gog_i| \to |\gog_{i+1}|$ being a weakly monotonic map such that $f_i(r_i) = r_{i+1}$ for all $i<I$ (we say that $f_i$ is {\em root preserving}), and ${\rm Id}_{I}$ is the identity on $|\gog_I|$.
\end{enumerate}
\end{defi}
%\vspace{-2mm}
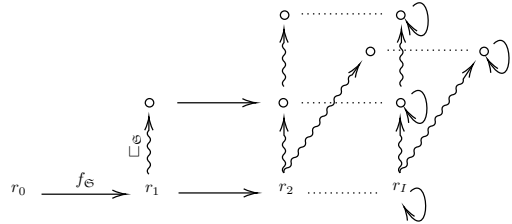
\begin{figure}[H]
%\captionsetup{width=0.45\textwidth}
\centering
\scalebox{0.63}{%
\tikzset{every picture/.style={line width=0.75pt}} %set default line width to 0.75pt        

\begin{tikzpicture}[x=0.75pt,y=0.75pt,yscale=-1,xscale=1]
%uncomment if require: \path (0,270); %set diagram left start at 0, and has height of 270

%Curve Lines [id:da547839558416701] 
\draw    (458.5,199) .. controls (475.25,182.26) and (475.5,246.04) .. (454.48,212.6) ;
\draw [shift={(453.5,211)}, rotate = 419.26] [color={rgb, 255:red, 0; green, 0; blue, 0 }  ][line width=0.75]    (10.93,-3.29) .. controls (6.95,-1.4) and (3.31,-0.3) .. (0,0) .. controls (3.31,0.3) and (6.95,1.4) .. (10.93,3.29)   ;
%Straight Lines [id:da11364802077762093] 
\draw    (166.5,201) -- (231.5,201) ;
\draw [shift={(233.5,201)}, rotate = 180] [color={rgb, 255:red, 0; green, 0; blue, 0 }  ][line width=0.75]    (10.93,-3.29) .. controls (6.95,-1.4) and (3.31,-0.3) .. (0,0) .. controls (3.31,0.3) and (6.95,1.4) .. (10.93,3.29)   ;
%Straight Lines [id:da4455899670167772] 
\draw    (251.5,185) .. controls (249.83,183.33) and (249.83,181.67) .. (251.5,180) .. controls (253.17,178.33) and (253.17,176.67) .. (251.5,175) .. controls (249.83,173.33) and (249.83,171.67) .. (251.5,170) .. controls (253.17,168.33) and (253.17,166.67) .. (251.5,165) .. controls (249.83,163.33) and (249.83,161.67) .. (251.5,160) .. controls (253.17,158.33) and (253.17,156.67) .. (251.5,155) .. controls (249.83,153.33) and (249.83,151.67) .. (251.5,150) -- (251.5,149) -- (251.5,141) ;
\draw [shift={(251.5,139)}, rotate = 450] [color={rgb, 255:red, 0; green, 0; blue, 0 }  ][line width=0.75]    (10.93,-3.29) .. controls (6.95,-1.4) and (3.31,-0.3) .. (0,0) .. controls (3.31,0.3) and (6.95,1.4) .. (10.93,3.29)   ;
%Straight Lines [id:da2679654943131342] 
\draw    (358.5,184) .. controls (356.83,182.33) and (356.83,180.67) .. (358.5,179) .. controls (360.17,177.33) and (360.17,175.67) .. (358.5,174) .. controls (356.83,172.33) and (356.83,170.67) .. (358.5,169) .. controls (360.17,167.33) and (360.17,165.67) .. (358.5,164) .. controls (356.83,162.33) and (356.83,160.67) .. (358.5,159) .. controls (360.17,157.33) and (360.17,155.67) .. (358.5,154) -- (358.5,150) -- (358.5,142) ;
\draw [shift={(358.5,140)}, rotate = 450] [color={rgb, 255:red, 0; green, 0; blue, 0 }  ][line width=0.75]    (10.93,-3.29) .. controls (6.95,-1.4) and (3.31,-0.3) .. (0,0) .. controls (3.31,0.3) and (6.95,1.4) .. (10.93,3.29)   ;
%Straight Lines [id:da7911166466053604] 
\draw    (450.5,184) .. controls (448.83,182.33) and (448.83,180.67) .. (450.5,179) .. controls (452.17,177.33) and (452.17,175.67) .. (450.5,174) .. controls (448.83,172.33) and (448.83,170.67) .. (450.5,169) .. controls (452.17,167.33) and (452.17,165.67) .. (450.5,164) .. controls (448.83,162.33) and (448.83,160.67) .. (450.5,159) .. controls (452.17,157.33) and (452.17,155.67) .. (450.5,154) -- (450.5,150) -- (450.5,142) ;
\draw [shift={(450.5,140)}, rotate = 450] [color={rgb, 255:red, 0; green, 0; blue, 0 }  ][line width=0.75]    (10.93,-3.29) .. controls (6.95,-1.4) and (3.31,-0.3) .. (0,0) .. controls (3.31,0.3) and (6.95,1.4) .. (10.93,3.29)   ;
%Straight Lines [id:da7835900352407431] 
\draw    (252,128) ;
\draw [shift={(252,128)}, rotate = 0] [color={rgb, 255:red, 0; green, 0; blue, 0 }  ][line width=0.75]      (0, 0) circle [x radius= 3.35, y radius= 3.35]   ;
%Straight Lines [id:da6872612660698377] 
\draw    (358,128) ;
\draw [shift={(358,128)}, rotate = 0] [color={rgb, 255:red, 0; green, 0; blue, 0 }  ][line width=0.75]      (0, 0) circle [x radius= 3.35, y radius= 3.35]   ;
%Straight Lines [id:da6997144973842714] 
\draw    (358.5,115) .. controls (356.83,113.33) and (356.83,111.67) .. (358.5,110) .. controls (360.17,108.33) and (360.17,106.67) .. (358.5,105) .. controls (356.83,103.33) and (356.83,101.67) .. (358.5,100) .. controls (360.17,98.33) and (360.17,96.67) .. (358.5,95) .. controls (356.83,93.33) and (356.83,91.67) .. (358.5,90) .. controls (360.17,88.33) and (360.17,86.67) .. (358.5,85) .. controls (356.83,83.33) and (356.83,81.67) .. (358.5,80) -- (358.5,79) -- (358.5,71) ;
\draw [shift={(358.5,69)}, rotate = 450] [color={rgb, 255:red, 0; green, 0; blue, 0 }  ][line width=0.75]    (10.93,-3.29) .. controls (6.95,-1.4) and (3.31,-0.3) .. (0,0) .. controls (3.31,0.3) and (6.95,1.4) .. (10.93,3.29)   ;
%Straight Lines [id:da25289830005481706] 
\draw    (359,58) ;
\draw [shift={(359,58)}, rotate = 0] [color={rgb, 255:red, 0; green, 0; blue, 0 }  ][line width=0.75]      (0, 0) circle [x radius= 3.35, y radius= 3.35]   ;
%Straight Lines [id:da7120467941937912] 
\draw    (451,128) ;
\draw [shift={(451,128)}, rotate = 0] [color={rgb, 255:red, 0; green, 0; blue, 0 }  ][line width=0.75]      (0, 0) circle [x radius= 3.35, y radius= 3.35]   ;
%Straight Lines [id:da6310567643005955] 
\draw    (451.5,115) .. controls (449.83,113.33) and (449.83,111.67) .. (451.5,110) .. controls (453.17,108.33) and (453.17,106.67) .. (451.5,105) .. controls (449.83,103.33) and (449.83,101.67) .. (451.5,100) .. controls (453.17,98.33) and (453.17,96.67) .. (451.5,95) .. controls (449.83,93.33) and (449.83,91.67) .. (451.5,90) .. controls (453.17,88.33) and (453.17,86.67) .. (451.5,85) .. controls (449.83,83.33) and (449.83,81.67) .. (451.5,80) -- (451.5,79) -- (451.5,71) ;
\draw [shift={(451.5,69)}, rotate = 450] [color={rgb, 255:red, 0; green, 0; blue, 0 }  ][line width=0.75]    (10.93,-3.29) .. controls (6.95,-1.4) and (3.31,-0.3) .. (0,0) .. controls (3.31,0.3) and (6.95,1.4) .. (10.93,3.29)   ;
%Straight Lines [id:da22338390527639806] 
\draw    (451,58) ;
\draw [shift={(451,58)}, rotate = 0] [color={rgb, 255:red, 0; green, 0; blue, 0 }  ][line width=0.75]      (0, 0) circle [x radius= 3.35, y radius= 3.35]   ;
%Straight Lines [id:da8332381081805674] 
\draw    (358.5,184) .. controls (358.07,181.68) and (359.02,180.31) .. (361.34,179.88) .. controls (363.66,179.46) and (364.61,178.09) .. (364.18,175.77) .. controls (363.75,173.45) and (364.7,172.08) .. (367.02,171.65) .. controls (369.33,171.22) and (370.28,169.85) .. (369.85,167.54) .. controls (369.42,165.22) and (370.37,163.85) .. (372.69,163.42) .. controls (375.01,162.99) and (375.96,161.62) .. (375.53,159.3) .. controls (375.1,156.98) and (376.05,155.61) .. (378.37,155.19) .. controls (380.69,154.76) and (381.64,153.39) .. (381.21,151.07) .. controls (380.78,148.75) and (381.73,147.38) .. (384.05,146.96) .. controls (386.37,146.53) and (387.32,145.16) .. (386.89,142.84) .. controls (386.46,140.52) and (387.41,139.15) .. (389.73,138.72) .. controls (392.04,138.29) and (392.99,136.92) .. (392.56,134.61) .. controls (392.13,132.29) and (393.08,130.92) .. (395.4,130.49) .. controls (397.72,130.07) and (398.67,128.7) .. (398.24,126.38) .. controls (397.81,124.06) and (398.76,122.69) .. (401.08,122.26) .. controls (403.4,121.83) and (404.35,120.46) .. (403.92,118.14) .. controls (403.49,115.82) and (404.44,114.45) .. (406.76,114.03) .. controls (409.08,113.6) and (410.03,112.23) .. (409.6,109.91) .. controls (409.17,107.59) and (410.11,106.22) .. (412.43,105.79) -- (412.82,105.23) -- (417.36,98.65) ;
\draw [shift={(418.5,97)}, rotate = 484.59] [color={rgb, 255:red, 0; green, 0; blue, 0 }  ][line width=0.75]    (10.93,-3.29) .. controls (6.95,-1.4) and (3.31,-0.3) .. (0,0) .. controls (3.31,0.3) and (6.95,1.4) .. (10.93,3.29)   ;
%Straight Lines [id:da747413415224292] 
\draw    (450.5,184) .. controls (450.03,181.69) and (450.96,180.3) .. (453.27,179.84) .. controls (455.58,179.38) and (456.51,177.99) .. (456.05,175.68) .. controls (455.58,173.37) and (456.51,171.98) .. (458.82,171.52) .. controls (461.13,171.06) and (462.06,169.67) .. (461.59,167.36) .. controls (461.13,165.05) and (462.06,163.66) .. (464.37,163.2) .. controls (466.68,162.74) and (467.61,161.35) .. (467.14,159.04) .. controls (466.67,156.73) and (467.6,155.34) .. (469.91,154.88) .. controls (472.22,154.42) and (473.15,153.03) .. (472.69,150.72) .. controls (472.22,148.41) and (473.15,147.02) .. (475.46,146.56) .. controls (477.77,146.1) and (478.7,144.71) .. (478.24,142.4) .. controls (477.77,140.09) and (478.7,138.7) .. (481.01,138.24) .. controls (483.32,137.78) and (484.25,136.39) .. (483.78,134.08) .. controls (483.32,131.77) and (484.25,130.38) .. (486.56,129.92) .. controls (488.87,129.46) and (489.8,128.07) .. (489.33,125.76) .. controls (488.86,123.45) and (489.79,122.06) .. (492.1,121.6) .. controls (494.41,121.14) and (495.34,119.75) .. (494.88,117.44) .. controls (494.41,115.13) and (495.34,113.74) .. (497.65,113.28) .. controls (499.96,112.82) and (500.89,111.43) .. (500.42,109.12) -- (502.95,105.32) -- (507.39,98.66) ;
\draw [shift={(508.5,97)}, rotate = 483.69] [color={rgb, 255:red, 0; green, 0; blue, 0 }  ][line width=0.75]    (10.93,-3.29) .. controls (6.95,-1.4) and (3.31,-0.3) .. (0,0) .. controls (3.31,0.3) and (6.95,1.4) .. (10.93,3.29)   ;
%Straight Lines [id:da0068725006840475444] 
\draw  [dash pattern={on 0.84pt off 2.51pt}]  (377.5,200) -- (432.5,200) ;
%Straight Lines [id:da24203778073087068] 
\draw  [dash pattern={on 0.84pt off 2.51pt}]  (372.5,57) -- (437.5,57) ;
%Straight Lines [id:da7869695641931013] 
\draw  [dash pattern={on 0.84pt off 2.51pt}]  (372.5,128) -- (437.5,128) ;
%Straight Lines [id:da35836001740469714] 
\draw    (427,87) ;
\draw [shift={(427,87)}, rotate = 0] [color={rgb, 255:red, 0; green, 0; blue, 0 }  ][line width=0.75]      (0, 0) circle [x radius= 3.35, y radius= 3.35]   ;
%Straight Lines [id:da736560464393656] 
\draw    (517,87) ;
\draw [shift={(517,87)}, rotate = 0] [color={rgb, 255:red, 0; green, 0; blue, 0 }  ][line width=0.75]      (0, 0) circle [x radius= 3.35, y radius= 3.35]   ;
%Straight Lines [id:da16708823622984592] 
\draw  [dash pattern={on 0.84pt off 2.51pt}]  (440.5,86) -- (505.5,86) ;
%Straight Lines [id:da4292430420446274] 
\draw    (274,128) -- (334.5,128) ;
\draw [shift={(336.5,128)}, rotate = 180] [color={rgb, 255:red, 0; green, 0; blue, 0 }  ][line width=0.75]    (10.93,-3.29) .. controls (6.95,-1.4) and (3.31,-0.3) .. (0,0) .. controls (3.31,0.3) and (6.95,1.4) .. (10.93,3.29)   ;
%Curve Lines [id:da8937223185927173] 
\draw    (524.5,81) .. controls (541.25,64.26) and (541.5,128.04) .. (520.48,94.6) ;
\draw [shift={(519.5,93)}, rotate = 419.26] [color={rgb, 255:red, 0; green, 0; blue, 0 }  ][line width=0.75]    (10.93,-3.29) .. controls (6.95,-1.4) and (3.31,-0.3) .. (0,0) .. controls (3.31,0.3) and (6.95,1.4) .. (10.93,3.29)   ;
%Curve Lines [id:da24365696897157518] 
\draw    (460.5,51) .. controls (477.25,34.25) and (477.5,98.04) .. (456.48,64.6) ;
\draw [shift={(455.5,63)}, rotate = 419.26] [color={rgb, 255:red, 0; green, 0; blue, 0 }  ][line width=0.75]    (10.93,-3.29) .. controls (6.95,-1.4) and (3.31,-0.3) .. (0,0) .. controls (3.31,0.3) and (6.95,1.4) .. (10.93,3.29)   ;
%Curve Lines [id:da14410277362484925] 
\draw    (459.5,121) .. controls (476.25,104.26) and (476.5,168.04) .. (455.48,134.6) ;
\draw [shift={(454.5,133)}, rotate = 419.26] [color={rgb, 255:red, 0; green, 0; blue, 0 }  ][line width=0.75]    (10.93,-3.29) .. controls (6.95,-1.4) and (3.31,-0.3) .. (0,0) .. controls (3.31,0.3) and (6.95,1.4) .. (10.93,3.29)   ;
%Straight Lines [id:da14397050261023692] 
\draw    (275,200) -- (335.5,200) ;
\draw [shift={(337.5,200)}, rotate = 180] [color={rgb, 255:red, 0; green, 0; blue, 0 }  ][line width=0.75]    (10.93,-3.29) .. controls (6.95,-1.4) and (3.31,-0.3) .. (0,0) .. controls (3.31,0.3) and (6.95,1.4) .. (10.93,3.29)   ;

% Text Node
\draw (247,192.4) node [anchor=north west][inner sep=0.75pt]    {$r_{1}$};
% Text Node
\draw (353,191.4) node [anchor=north west][inner sep=0.75pt]    {$r_{2}$};
% Text Node
\draw (443,190.4) node [anchor=north west][inner sep=0.75pt]    {$r_{I}$};
% Text Node
\draw (141,192.4) node [anchor=north west][inner sep=0.75pt]    {$r_{0}$};
% Text Node
\draw (192,182.4) node [anchor=north west][inner sep=0.75pt]    {$f_{\mathfrak S }$};
% Text Node
\draw (234.5,173.5) node [anchor=north west][inner sep=0.75pt]  [rotate=-270]  {$\sqsubset _{\mathfrak S }$};

\end{tikzpicture}
}
\vspace{-2mm}
\caption[story]{An example of a $\bold{GL}$-story. The squiggly arrows represent the relation $\sqsubset_{\mathfrak  S}$ while the straight  arrows represent the function $f_{\mathfrak  S}$. Each vertical slice represents a $\bold{GL}$-moment. In the case of other types of stories, we may also have clusters besides singletons.}
\end{figure}
We often omit the subindices $\mathfrak m$ or $\gog$ when this does not lead to confusion.
We may also assign different notations to the components of a moment, so that if we write $\mathfrak m= \langle W,\sqsubset   ,\nu,x\rangle $, it is understood that $W=|\mathfrak m|$, ${\sqsubset  }= {\sqsubset  _\mathfrak m}$, etc.

Recall that a {\em $p$-morphism} between Kripke models is a type of map that preserves validity.
It can be defined in the context of dynamic derivative frames as follows:

\begin{defi}[dynamic $p$-morphism]
Let $\mathfrak M=\langle W_\mathfrak M, \sqsubset   _\mathfrak M, g_\mathfrak M  \rangle$ and $\mathfrak N=\langle W_\mathfrak N, \sqsubset   _\mathfrak N, g_\mathfrak N  \rangle$ be dynamic derivative frames. Let $\pi\colon W_\mathfrak M\to W_\mathfrak N$.
We say that $\pi$ is a {\em dynamic $p$-morphism} if $w\sqsubset  _\mathfrak M v$ implies that $\pi(w) \sqsubset  _\mathfrak N\pi(v)$, $\pi(w)\sqsubset  _\mathfrak N u$ implies that there is $v\sqsupset_\mathfrak M w$ with $\pi(v) = u$, and $\pi\circ g_\mathfrak M =g_\mathfrak N\circ \pi$.
\end{defi}

It is then standard that if $\pi\colon W_\mathfrak M\to W_\mathfrak N$ is a surjective, dynamic $p$-morphism, then any formula valid on $\mathfrak M$ is also valid on $\mathfrak N$.
However, our relation $\sqsubset  _{\Phi}$ will allow us to weaken these conditions and still obtain maps that preserve the truth of (some) formulas.

\begin{defi}[$\Phi$-morphism]\label{phimor}
Fix a logic $\Lambda$ and let $\mathfrak M^\Lambda_{\rm c} = \langle W_{\rm c},\sqsubset_{\rm c},g_{\rm c},\nu_{\rm c}\rangle$ and $\mathfrak S$ be a story of duration $I$.
A map ${\pi}:|{\gog}|\rightarrow W_{\rm c}$ is called a \emph{dynamic $\Phi$-morphism} if for all $x\in |{\gog}|$ the following conditions are satisfied:
\begin{enumerate}
\item $x\in {\nu_\mathfrak S} (p)  \iff p \in {\pi}(x)$;
\item\label{itCommutes}  If $x\in|\gog_i|$ for some $i<I$, then $g_{\rm c}( {\pi}(x))={\pi}( f_{\gog}(x))$;
\item \label{itForth} If $ x \sqsubset  _\gog y$ then $\pi(x) \sqsubset  _{\rm c}  \pi(y)$; %or $\pi(x), \pi(y) \in C(x_i)$ for some i;
\item\label{itBack} If ${\pi}(x)\sqsubset  _{\Phi}v$ for some $v\in W_{\rm c}$, then there exists $y\in|{\gog}|$ such that
$x \sqsubset  _\gog y  \text{ and } v={\pi}(y).$
\end{enumerate}
If we drop condition~\ref{itCommutes}, we say that $\pi$ is a {\em $\Phi$-morphism;} the latter notion will mostly be applied to moments, viewed as stories of duration one.
%If instead $x \prec y$ implies ${\pi}(x)\sqsubset  _{\rm c}{\pi}(y)$, then we call $\pi$ \emph{$\Phi$-morphism}.
\end{defi}

We now show that a dynamic $\Phi$-morphism $\pi$ preserves the truth of formulas of suitable $\bc$-depth, where the latter is defined as usual in terms of nested occurrences of $\bc$ in a formula $\varphi$.

\begin{restatable}[truth preservation]
{lem}{truth}Let $\mathfrak S$ be a story of duration $I$ and $x \in |\mathfrak S_0|$. Let $\pi$ be a dynamic $\Phi$-morphism to the canonical model of some normal logic $\Lambda$ over $\mathcal L^\circ_\dd$.
Suppose that $\varphi \in \Phi$ is a formula of $\bc$-depth at most $I$. Then $\varphi\in\pi(x)$ iff $x\in \| \varphi\|_\mathfrak S$.
\end{restatable}

\begin{proof}
We prove the more general claim that if $\varphi$ has $\bc$-depth at most $I-j$ and $x\in |\mathfrak S_j|$, then $\varphi\in\pi(x)$ iff $x\in \| \varphi\|_\mathfrak S$.
We proceed by induction on the complexity of $\varphi$.
For the base case, suppose $p\in \pi(x)$. Then by the definition of a dynamic $\Phi$-morphism we have $x\in\nu_\gog(p)$ as required. The other direction follows similarly, and since the Boolean connectives follow by standard arguments, then we focus on modalities.

If $\lozenge\varphi\in\pi(x)$ then there exists $v\in W_{\rm c}$ such that $\pi(x) \sqsubset  _{\Phi}  v$ and $\varphi\in v$. By the definition of a dynamic $\Phi$-morphism there exists $v'\in|{\gog}|$ such that $x \sqsubset_\gog v'$ and $v= \pi(v')$. By the induction hypothesis $v'\in\|\varphi\|_\gog$ and thus $x\in \|\lozenge\varphi\|_\gog$.

If $\lozenge\varphi\notin\pi(x)$ and $y\sqsupset_\mathfrak S x$ then $\pi(y)\sqsupset_\mathrm c \pi(x)$, so $\neg\varphi \in \pi(y)$ and the induction hypothesis and the fact that $y\in |\mathfrak S_j|$ (as $\mathfrak S_j$ is closed under $\sqsubset  _\mathfrak S$) yields $y  \in \|\neg\varphi\|_\mathfrak S$.
Since $y$ was arbitrary, $x  \in \|\neg\lozenge \varphi\|_\gog$.

%\item If $\bc\varphi\in t_\mathfrak S(w)$ then for some $v\in |\mathfrak S|$, $\varphi\in t_\mathfrak S(v)$. We get $g(v)\in V(\varphi)$ and since $f(g(w))=g(v)$ then $g(w)\in V(\bc\varphi)$.

If $\bc\varphi\in\pi(x)$ then $\varphi\in g_{\rm c}(\pi(x))$. By the definition of a dynamic $\Phi$-morphism we get $g_{\rm c}(\pi(x)) = \pi(f_\gog (x))$.
Note that $f_\gog (x) \in |\mathfrak S_{j+1}| $, but the $\bc$-depth of $\varphi$ is bounded by $I-j-1$.
From the induction hypothesis we get $ f_\gog (x)\in\| \varphi\|_\gog$. Hence by definition $x\in\| \bc\varphi\|_\gog$. 

If $x\in \| \bc\varphi\|_\gog$, then $f_\gog (x) \in \| \varphi\|_\gog$.
The bound on $\bc$-depth can be checked as above, so from the induction hypothesis it follows that $\varphi\in\pi(f_\gog (x))$ and by the definition of a dynamic $\Phi$-morphism we get $\pi(f_\gog(x))=g_{\rm c}(\pi(x))$. Therefore $\bc\varphi\in \pi(x)$.
\end{proof}

We next demonstrate that for every point $w$ in the canonical model, there exists a   $\Lambda$-moment $\mathfrak m$ for which there is a  $\Phi$-morphism mapping $\mathfrak m$ to $w$. 
In order to do this, we define a procedure for constructing new moments from smaller ones.

\begin{defi}[moment construction]
Let $\Lambda\in \{\bold{wK4},\bold{K4},\bold{GL}\}$ and $ {C}'\cup\{x\} \subseteq C(x )$ for some $x $ in the canonical model $\mathfrak{M}^\Lambda_\mathrm c$.
Let $\vec{\mathfrak{a}}=\lb \mathfrak{a}_m\rb_{m<N}$ be a sequence of moments.
We define a structure $\noment =  {{\vec{\mathfrak{a}}}\choose {C'}}_x$ as follows:
\begin{enumerate}
\item  $|\noment| = C'  \sqcup \bigsqcup_{m<N}| \mathfrak{a}_m | ;$
%\item $x\prec y$ if either $x\in C'_i$ and $y\in|\gog|_i$ or $x,y\in|\mathfrak{a}^m|$ and $x\prec_{\mathfrak{a}^m}y$;

\item $y\sqsubset  _\noment z $ if either
\begin{itemize}
\item $y,z \in C' $, $\Lambda\neq\bold{GL}$ and $y\sqsubset  _{\rm c}z$,

\item  $y \in C' $ and $z \in |\mathfrak{a}_m| $ for some $m$, or

\item $y,z\in|\mathfrak{a}_m|$ and $y\sqsubset  _{\mathfrak{a}_m} z$ for some $m$;
\end{itemize}
\item $\nu_\noment (p)=\{x\in C':  x\in \nu_{\rm c}(p)\} 
\sqcup\bigsqcup_{m<N}
\nu_{\mathfrak{a}_m}(p)$;

\item $r_\noment = x$.
\end{enumerate}
\end{defi}

We remark that $N=0$ is allowed, in which case the construction outputs only the root cluster.
The moment construction for some logic $\Lambda$ can be used to produce $\Lambda$-moments.

\begin{lem}\label{storyide}
Let $\Lambda\in \{\bold{wK4},\bold{K4},\bold{GL}\}$ and $ {C}' \subseteq C(x )$ for some $x $ in the canonical model $\mathfrak{M}^\Lambda_\mathrm c$, where $C'$ is a singleton if $\Lambda = \bold{GL}$ and finite otherwise.
Let $\vec{\mathfrak{a}}=\lb \mathfrak{a}_m\rb_{m<N}$ be a sequence of $\Lambda$-moments.
Then, ${\vec{\mathfrak{a}}\choose C}_x$ is a $\Lambda$-moment.
\end{lem}

\begin{proof}
If $\Lambda= \bold{wK4}$ then $\sqsubset  $ is easily seen to be weakly transitive since each $\mathfrak{a}_i$ is weakly transitive and the root sees all other points. The transitivity of $\sqsubset  $ where $\Lambda\supseteq \bold{K4}$ follows by the same reasoning. The irreflexivity of $\sqsubset  $ where $\Lambda= \bold{GL}$ follows from the fact that $C'$ is an irreflexive singleton and each $\mathfrak{a}_i$ is irreflexive.
 \end{proof}

 We use this moment construction to show the existence of a $\Phi$-morphism between a moment and the canonical model.
 Below, we denote by $C_\Phi(w)$ the $\sqsubset  _{\Phi}$-cluster of $w$, i.e.\
$$C_\Phi(w)=\{w\}\cup \{v:w\sqsubset  _{\Phi}v \sqsubset  _{\Phi} w\}.$$
\begin{restatable}{lem}{existsmoment}\label{lemmExistsMoment}
Given $\Lambda \in \{\mathbf{wK4C},\mathbf{K4C},\mathbf{GLC} \}$, for all $w\in W_\mathrm c$ there exists a $\Lambda$-moment $\moment$ and a $\Phi$-morphism $\pi\colon |\moment|\to W_\mathrm c$ such that $\pi(r_\moment) = w$.
\end{restatable}

\begin{proof}

We prove the stronger claim, that there is a moment $ \moment$ and a map $ \pi\colon | \moment| \to W_{\rm c} $ that is a $p$-morphism to the structure $
\langle W_{\rm c},\sqsubset_\Phi\rangle$ (we will say that $\pi$ is a {\em $p$-morphism with respect to $\sqsubset_\Phi$}).
Let $\sqsubset  _{\Phi}^1$ be the strict $\sqsubset  _{\Phi}$ successor relation, i.e.\ $w\sqsubset  _{\Phi}^1v$ iff $w \sqsubset  _{\Phi} v$ and $\neg (v \sqsubset  _{\Phi} w)$.
Since $\sqsubset  _{\Phi}^1$ is converse well-founded, we can assume inductively that for each $v$ such that $w \sqsubset  _{\Phi}^1 v$, there is a moment ${\moment_{v}}$ and a $p$-morphism $\pi_{v}:|{\moment_{v}}|\rightarrow W_{\rm c}$ with respect to $\sqsubset  _{\Phi}$ that maps the root of ${\moment_{v}}$ to $v$.
Accordingly, we define a moment
 $$ \moment = {\{{\moment}_v: v \sqsupset  _{\Phi}^1 w \}\choose C_\Phi(w) }_w;$$
note that $C_\Phi(w)$ is finite and, if $\Lambda=\bf{GLC}$, it is a singleton, so $\moment$ is a $\Lambda$-moment by Lemma~\ref{storyide}.

Next we define a map $ \pi:|\moment |\rightarrow W$ as
$$   \pi(x)=\begin{cases} x &\text{if }x \in C_\Phi(w), \\
\pi_v (x)&\text{if }x\in|\moment _v|. \end{cases}$$

We prove that $  \pi$ is a $p$-morphism for $\sqsubset_\Phi$.
First assume that $  \pi(x)\sqsubset  _{\Phi} v$; we must show that there is $y\sqsupset_\moment x$ so that $ \pi(y) = v $.
Either $x\in C_\Phi(w)$ or $x\in |\moment_u|$ for some $u$.
In the first case, we consider two sub-cases.
If $v\in C_\Phi(w)$ as well, then $ \pi(v) = v$ and we may set $y:=v$.
If not, $   \pi(x)\sqsubset  _{\Phi}^1 v$, and by letting $y$ be the root of $\moment_v$, we see that $ \pi(y) = v$.
If instead $x\in |\moment_u|$ for some $u$, then by assumption $\pi_u$ is a $p$-morphism, immediately yielding the desired $y$ (also satisfying $y \in |\moment_u|$).

Now, suppose $x \sqsubset  _\moment y$. We check that $  \pi(x)\sqsubset  _\Phi   \pi(y)$.
There are two cases to consider:
first suppose that $x\in | \moment_{u}|$ for some $u$ and $y\in |\moment_{v}|$ for some $v$. Then by the definition of the moment construction operation, $ u = v $. By the induction hypothesis, since $ \pi_{u}$ is a $p$-morphism, then $  \pi(x)\sqsubset  _\Phi  \pi(y)$.
Otherwise, suppose that $x\in C_\Phi(w)$.
If $y\in C_\Phi(w)$ we immediately obtain $\pi(x)\sqsubset  _\Phi  \pi(y)$, since $C_\Phi(w)$ is precisely the $\sqsubset  _\Phi$-cluster of $w$ (with the same accessibility relation).
Otherwise, $y\in |\moment_{v}|$ for some $v$, and we let $y'$ be the root of $\moment_v$, so that $ \pi(y') = v$.
Then,  $  \pi(x) \sqsubset  _{\Phi}  \pi(y')$.
If $y=y'$ we are done.
If $y\neq y'$, then since by the induction hypothesis $  \pi_{v}$ is a $p$-morphism with respect to $\sqsubset_\Phi$, then $y'  \sqsubset   y$ implies $  \pi(y') \sqsubset  _ \Phi \pi(y)$. By the (weak) transitivity of $\sqsubset  _\Phi$ we have that either $ \pi(x) \sqsubset  _\Phi  \pi(y)$ and we are done, or if $\Lambda=\bf wK4C$, possibly $  \pi(x) =   \pi(y)$.
But then, $  \pi(y') \sqsubset  _ \Phi \pi(x)$, contradicting that $ \pi(x) \sqsubset^1_\Phi \pi(y ')$.
\end{proof}

Next, we wish to extend Lemma~\ref{lemmExistsMoment} to provide a dynamic $\Phi$-morphism instead of a `static' one.
This will be achieved via a step-by-step construction.
The issue is that the construction will not yield $\Phi$-morphisms, but rather an approximate version which we call \emph{pre-$\Phi$-morphisms.}
Basically, the worlds `at the bottom' of the moments we construct will not be well behaved, but these defects can be fixed using a quotient which affects only these worlds at the bottom.
The following definitions make these notions precise.

\begin{defi}[$\Lambda$-bottom]
Let $  \moment$ be a moment and $ \pi :  | \moment|\rightarrow W_{\rm c}$.
We say that $x\in |\moment|$ is \emph{at the $\Lambda$-bottom} for $\Lambda\in\{\bold{wK4},\bold{K4},\bold{ GL}\}$ if
\begin{itemize}
\item $\Lambda \neq \bold{GL}$ and $\pi (x)\in C(\pi(r_{\moment}))$, or

\item $\Lambda=\bold{GL}$ and for all $y \sqsubset  _\moment x$, $\pi (y)=\pi (x)$.
\end{itemize}
%$$ \forall y\preceq x\exists z\simeq y(\hat\hat\pi(z)\in C'(x_i)).$$
We will refer to `$\Lambda$-bottom' simply as `bottom' when this does not lead to confusion.
\end{defi}

We say that $ \pi:|\moment|\rightarrow W_{\rm c}$ is a \emph{pre-$\Phi$-morphism} if it fulfils conditions 1 and 4 of a $\Phi$-morphism (Definition~\ref{phimor}), and
$ x\sqsubset  _\mathfrak m y $ implies that either $ \pi (x) \sqsubset  _{\rm c}  \pi(y) $ or $ x,y $ are at the bottom.

\begin{defi}[quotient moment] %Let $x\simeq y$ if $x\preceq y$ and $y\preceq x$.
Let $\Lambda\in\{\bold{wK4C},\bold{K4C},\bold{ GLC}\}$.
Let $  \moment$ be a $\Lambda$-moment and $ \pi :  | \moment|\rightarrow W_\mathrm c$ be a pre-$\Phi$-morphism.
We define $x\sim y$ if either $x=y$, or $x,y$ are at the bottom and $\pi (x) = \pi (y)$.
Given $y\in|\moment|$ we set $[y]=\{z:z\sim y\}$.

The \emph{quotient moment} $ \quot \moment \pi $ of $\moment$ and its respective map $\quotpi \pi:|\quot\moment \pi|\rightarrow W_{\rm c}$ are defined as follows, where $x,y\in|\moment|$:
\begin{enumerate}[(a)]
\item $|\quot\moment\pi|=\{[x]:x\in| {\moment}|\}$;
\item $[x]\sqsubset  _{\quot\moment\pi} [y]$ iff one of the following conditions is satisfied:
\begin{itemize}
\item  $x,y$ are at the bottom, $ \pi(x) \sqsubset  _{\rm c}  \pi(y)$, and $\Lambda\neq \bold{GLC}$;
\item   $x$ is at the bottom and $y$ is not at the bottom;
\item $x,y$ are not at the bottom and $x \sqsubset  _\moment y$;
\end{itemize}

\item $\nu_{\quot\moment\pi} (p)=\{[x]:x\in {\nu_\moment}(p)\}$;

\item $r_{\quot\moment\pi} = [r_\moment] $;

\item $\quotpi\pi ([x])= {\pi}(x)$.
\end{enumerate} 
\end{defi}

The above is well-defined due to the way we defined $\sim$; for example, (b) is well-defined since $x\sim x'$ implies that either both are at the bottom and neither is, and in the first bullet point we use that $x\sim x'$ implies $\pi(x) = \pi(x')$ (and similarly for $y\sim y'$).
We also use this property to see that (c) is well-defined, since $\pi(x) = \pi(x')$ implies that the two agree on all atoms.

The following proposition shows that the property of being at the bottom is downward-closed.

 \begin{prop}\label{atthebottom}
Let $  \moment$ be a moment and $ \pi :  | \moment|\rightarrow W_\mathrm c$ be a pre-$\Phi$-morphism.
Let $x,y\in |\moment|$. If $x$ is at the $\Lambda$-bottom and $y\sqsubset  _\moment x$, then $y$ is at the $\Lambda$-bottom. \end{prop}

\begin{proof}

%Suppose $x$ is at the $\bold{wK4C}$-bottom and $y\hat \prec x$. Then by definition $y\in C'_i$ and so $\hat\pi(y)\in C'_i$ and $y$ is at the $\bold{wK4C}$-bottom. Suppose $x$ is at the $\bold{GLC}$-bottom. Since $y\hat\prec x$ then $y\neq x$ since $\sqsubset  _{\Phi}$ is irreflexive. But then by definition $y\in C'_i$ in contradiction to the fact that $C'_i$ is a singleton. The statement then follows trivially. 
Suppose $\Lambda \neq \bold{GLC}$. If $x$ is at the bottom, we have that $\pi(x)\in C(\pi(r_\moment))$.
If moreover $y\sqsubset  _\moment x$, then since $\pi$ is a pre-$\Phi$-morphism, we have that either $\pi(y)\sqsubset  _{\rm c}\pi(x)$ or $y$ is at the bottom.
In the latter case there is nothing to prove.
Otherwise, suppose $r_\moment \sqsubset  _\moment y \sqsubset  _\moment x$ and $y$ is not at the bottom. Then $\pi(r_\moment) \sqsubset  _{\rm c}  \pi(y) \sqsubset  _{\rm c} \pi( x) $ and hence $\pi(y) \in C(r_\moment)$.
This is a  contradiction and so $y$ is at the bottom.

Suppose $\Lambda=\bold{GLC}$. Then by the transitivity of $\sqsubset  _\moment$ it follows that $y$ is also at the bottom.
\end{proof}

The quotient moment and its respective map hold some essential properties for the derivation of this section's main result.
The idea is that by constructing a $\Lambda$-moment $m$ with an associated pre-$\Phi$-morphism $\pi$, we get that $\quot\moment\pi$ is still a $\Lambda$-moment, but now $[\pi]$ is a proper $\Phi$-morphism. The next few results make this intuition precise.

\begin{restatable}{lem}{premor}\label{lemmPremor}
For $\Lambda \in \{\bold{wK4},\bold{K4},\bold{GL}\}$, if $\moment$ is any $\Lambda$-moment and $ \pi :  | \moment|\rightarrow W_{\rm c}$ is a pre-$\Phi$-morphism, then $\quot\moment\pi$ is a $\Lambda$-moment and $[\cdot]\colon |\moment| \to |\quot\moment\pi|$ is weakly monotonic and root-preserving.
\end{restatable}

\begin{proof}
For the cases where $\Lambda=\bold{wK4}$ and $\Lambda=\bold{K4}$ it is easy to verify that  $|\quot\moment\pi|$ is weakly transitive and transitive, respectively. For the case where $\Lambda=\bold{GL}$ suppose for contradiction that $[x] \sqsubset  _{\quot\moment\pi}[x]$. Then it must be the case that $x$ is not at the bottom and $x\sqsubset  _\moment x$ in contradiction to $\mathfrak m$ being a $\bold{GL}$-moment.

Next we check that $[\cdot]$ is weakly monotonic.
Suppose that $x\sqsubset  _{\moment}y$.
By Proposition~\ref{atthebottom}, if $ y $ is at the bottom, so is $ x $.
Thus in order to prove that $[ x ]\sqsubseteq_{\quot\moment\pi}[ y ]$, there are three cases to consider.
If $ x , y $ are at the bottom, either $\pi( x ) = \pi( y )$, so that $[ x ] = [ y ]$, or else $\Lambda\neq\mathbf{GLC}$ and $\pi( x ) \sqsubset  _\mathrm c \pi( y )$, hence $[ x ] \sqsubset_{\quot\moment\pi}[ y ]$ by definition of $\sqsubset_{\quot\moment\pi}$.
	If $ x $ is at the bottom and $ y $ is not at the bottom, then by definition $[ x ] \sqsubset  _{\quot\moment\pi}[ y ] $.
	If neither is at the bottom, $  x  \sqsubset  _{ \moment }  y  $ yields $[ x ] \sqsubset  _{\quot\moment\pi}[ y ] $.
In all cases we have that $[ x ]\sqsubseteq_{\quot\moment\pi}[ y ]$.
\end{proof}

\begin{restatable}{prop}{welldefined}\label{propIsPMorph}
If $\moment$ is a moment and $\pi\colon |\moment| \to W_{\rm c}$ is a pre-$\Phi$-morphism, 
then the map $\quotpi\pi$ is a well-defined $\Phi$-morphism.
\end{restatable}

\begin{proof}
To show that $\quotpi\pi$ is well defined, suppose that for some $x,y\in|\moment|$ we have $x\neq y$ and $[x]=[y]$. By definition of $\sim$ it follows that $\pi(x)=\pi(y)$, as required. 
It remains to show that $[x]\sqsubset  _{\quot\moment\pi} [y]$ implies $\quotpi\pi([x])\sqsubset  _{\rm c}\quotpi\pi([y])$. Suppose $[x] \sqsubset  _{\quot\moment\pi} [y]$.
If both $x,y$ are at the bottom, then $[x] \sqsubset  _{\quot\moment\pi} [y]$ may only occur when $\pi(x) \sqsubset_{\rm c} \pi(y)$.
If neither $x$ nor $y$ are at the bottom then $[x] \sqsubset  _{\quot\moment\pi} [y]$ may only occur when $ x \sqsubset_{\moment} y$, which implies that $\pi(x) \sqsubset_{\rm c} \pi(y)$ since $\pi$ is a pre-$\Phi$-morphism.
The case where $y$ is at the bottom but $x$ is not is impossible by Lemma~\ref{atthebottom}.
So we are left with the case where $x$ is at the bottom but $y$ is not.
But then $\pi(x) = \pi(r_\mathfrak m) \sqsubset_{\rm c} \pi(y)$, where $\pi(r_\mathfrak m) \sqsubset_{\rm c} \pi(y)$ is obtained from the definition of pre-$\Phi$-morphism, using the fact that $r_\mathfrak m \sqsubset_\mathfrak m y$ and $y$ is not at the bottom.
We conclude that $\quotpi\pi$ satisfies Condition~\ref{itForth} of a $\Phi$-morphism.

For Condition~\ref{itBack}, suppose that $\quotpi\pi([x]) \sqsubset_\Phi v$.
Since $\quotpi\pi([x]) = \pi(x)$, there is $y\sqsupset_\mathfrak m x$ such that $\pi(y) = v$.
Inspecting the definition of $\sqsubset_{\quot\moment\pi}$ yields that $\quotpi\pi([x])\sqsubset_{\quot\moment\pi}\quotpi\pi([y])$, except possibly when $\Lambda = \bf{GLC}$ and $x,y$ are at the bottom.
However, in this case we would have $v=\pi(x)$, but $ \sqsubset_\Phi   $ is irreflexive, so this is impossible.
We conclude that Condition~\ref{itBack} is also satisfied.
 \end{proof}
 
 Thus from a pre-$\Phi$-morphism we can produce a suitable $\Phi$-morphism. It remains to provide tools for producing dynamic $\Phi$-morphisms.
 
 \begin{prop}\label{propQuotMap}
Fix $\Lambda \in \{\mathbf{wK4C},\mathbf{K4C},\mathbf{GLC}\}$ and let $\mathfrak M^\Lambda_{\rm c} = \langle W,\sqsubset,g,\nu\rangle$.
Let $\moment$ and $\noment$ be moments, $ f \colon |\moment|\to|\noment|$ be root-preserving, weakly monotonic, and $\pi \colon |\moment| \to W$ and $\rho \colon |\noment| \to W $ be such that $\pi$ is a $\Phi$-morphism, $\rho$ is a pre-$\Phi$-morphism, and $g\circ \pi = \rho\circ f$.
Let $\quotpi f$ be given by $\quotpi f( x ) = [f(x)]$.
Then, $\quotpi f \colon |\moment| \to |\quot\noment \rho|$ is weakly monotonic and root-preserving and satisfies $g\circ \pi = \quotpi\rho\circ\quotpi  f $.
\end{prop}

\begin{proof}
Checking that $g\circ \pi = \quotpi\rho\circ\quotpi  f $ is routine, and $[f]$ is weakly monotonic and root-preserving since, writing $[f] = [\cdot]\circ f$, we see that it is a composition of maps with these properties in view of Lemma~\ref{lemmPremor}. 
\end{proof}
 
Using these properties, we can prove the existence of an appropriate story that maps to the canonical model.
This is based on the following useful lemma:

\begin{lem}\label{lemmExistsSuper}
Fix $\Lambda \in \{\mathbf{wK4C},\mathbf{K4C},\mathbf{GLC}\}$ and let $\mathfrak M^\Lambda_{\rm c} = \langle W_{\rm c},\sqsubset_{\rm c},g_{\rm c},\nu_{\rm c}\rangle$.
Let $\moment$ be a $\Lambda$-moment and suppose that there exists a $\Phi$-morphism $\pi\colon |\moment|\to W_{\rm c}$. Then, there exists a $\Lambda$-moment $ \noment$, a weakly monotonic map $f \colon |\moment| \to | \noment |$, and a $\Phi$-morphism $\rho\colon| \noment | \to W_{\rm c} $ such that
$g_{\rm c}\circ \pi = \rho \circ f$.
\end{lem}

\begin{proof}
We proceed by induction on the height of $\moment$.
Let $C$ be the cluster of $r_\moment$ and let $\vec {\mathfrak a} = \langle\mathfrak a_n\rangle_{n<N}$ be the generated sub-models of the immediate strict successors of $r_\moment$; note that each $\mathfrak a_n$ is itself a moment of smaller height, with $\Phi$-morphisms $\pi_n \colon |\mathfrak a_n| \to W_\mathrm c$ obtained by restricting $\pi$ to $|\mathfrak a_n|$ (note that $N=0$ is allowed if there are no strict successors).
By the induction hypothesis, there exist moments $\langle\mathfrak a'_n\rangle_{n<N}$, root-preserving, weakly monotonic maps $f_n \colon |\mathfrak a_n | \to |\mathfrak a'_n|$, and $\Phi$-morphism $\rho_n \colon |\mathfrak a'_n| \to W_{\rm c}$ such that  $g_{\rm c}\circ \pi_n = \rho_n \circ f_n$.
Moreover, for each $v\sqsupset  _{\Phi} g_{\rm c}(r_\moment) $, by Lemma~\ref{lemmExistsMoment} there are $\mathfrak b_v$ and a $\Phi$-morphism $\sigma_v \colon |\mathfrak b_v| \to W_{\rm c}$ mapping the root of $\mathfrak b_v$ to $v$.
Let $D = g_{\rm c} \pi (C) \cup C_\Phi(g_{\rm c}(r_\moment))$, and let $\hat\noment = {\vec{\mathfrak a} *\vec{\mathfrak b}\choose D}_{g_{\rm c}(w)}$.
Define maps  $\hat \rho\colon |\hat\noment| \to W_{\rm c}$ and $\hat f\colon |\moment| \to |\hat\noment|$, given by
\begin{multicols}2
$$
\hat \rho(w)
=
\begin{cases}
 w  &\text{if $w \in D$},\\
\rho_n(w) & \text{if $w \in |\mathfrak a_n|$},\\
\sigma_v(w) & \text{if $w \in |\mathfrak b_v|$};
\end{cases}
$$

$$
\hat f(w)
=
\begin{cases}
g_{\rm c}(w) &\text{if $w \in C$},\\
f_n(w) & \text{if $w \in |\mathfrak a_n|$}.
\end{cases}
$$
\end{multicols}
It is not hard to check that $\hat \rho$ is a pre-$\Phi$-morphism, $\hat f$ is weakly monotonic, and $g_{\rm c} \circ  \pi  = \hat \rho  \circ \hat f$.
Setting $\noment =\quot {\hat\noment}{\hat \rho}$, $f=\quotpi {\hat f}$ and $\rho =\quotpi {\hat \rho}$, Lemma~\ref{lemmPremor} tells us that $\noment$ is a $\Lambda$-moment and Propositions~\ref{propIsPMorph} and \ref{propQuotMap} imply that $f$ and $\rho$ have the desired properties.%\david{This might need more detail, and maybe be separated into a couple of lemmas.}
\end{proof}

\begin{restatable}{prop}{storyexist}\label{propExistsStory}
Fix $\Lambda\in \{\bold{wK4C},\bold{K4C},\bold{GLC}\}$.
Given $I<\omega$ and $w\in W_{\rm c}$, there is a story $\gog$ of duration $I$ and a dynamic $\Phi$-morphism $\pi\colon |\gog| \to W_{\rm c}$ with $w =\pi(r_\gog)$.
\end{restatable}

\begin{proof}
Proceed by induction on $I$.
For $I=0$, this is essentially Lemma~\ref{lemmExistsMoment}.
Otherwise, by the induction hypothesis, assume that a story $\hat{\gog}$ of depth $I$ and dynamic $p$-morphism $\hat \pi$ exist. By Lemma~\ref{lemmExistsSuper}, there is a moment $\gog_{I+1}$, map $f_I\colon |\gog_I| \to |\gog_{I+1}|$, and $\Phi$-morphism $\pi_{I+1} \colon |\gog_{I+1}| \to W_{\rm c}$ commuting with $f_I$.
We define $\gog$ by adding
%\david{I think it is pretty obvious what is meant here, but we can give a proper definition.}
$\gog_{I+1}$ to $\hat{\gog}$ in order to obtain the desired story.
\end{proof}

 It follows that any satisfiable formula is also satisfiable on a finite story, hence satisfiable on a finite model, yielding the main result of this section.

\begin{restatable}{thm}{mainkripke}\label{wk4c}
The logics $\bold{wK4C}$, $\bold{K4C}$ and $\bold{GLC}$ are sound and complete for their respective class of finite dynamic $\Lambda$-frames.
\end{restatable}

\begin{proof}
Soundness follows from Lemma~\ref{lemCH} and the well-known soundness results reviewed in Section~\ref{secDTL}.
Let $\Lambda\in\{\bold{wK4C},\bold{K4C},\bold{GLC}\}$ and suppose $\Lambda\nvdash\varphi$. Then in the canonical model $\mathfrak{M}^\Lambda_{\rm c}=\langle W,\sqsubset,g,\nu\rangle $,
%\yoav{I changed here from $V$ to $\nu$}
there is $w\in W$ that refutes $\varphi$. 
Let $\Phi$ be the set of subformulas of $\varphi$.
By Proposition~\ref{propExistsStory}, there is a story $\gog$ and a dynamic $\Phi$-morphism $\pi:|\gog|\rightarrow W$ such that $w = \pi(r)$, where $r$ is a root of $\gog$. 
It follows that $\gog,r \not\models\varphi$. Recall that $\gog$ is a finite dynamic derivative frame.
In the case of $\bold{K4C}$, $\gog$ is also transitive, and in the case of $\bold{GLC}$, also transitive and irreflexive.
Then $\bold{wK4C}$ is complete with respect to finite dynamic derivative frames, $\bold{K4C} $ with respect to finite, transitive, dynamic derivative frames, and $\bold{GLC}$ is complete with respect to finite, irreflexive, transitive dynamic derivative frames.
\end{proof}
\section{Topological $d$-completeness}
In this section we establish completeness results for classes of dynamic topological systems with continuous functions.
We begin with the logic $\bold{GLC}$, simply because the topological $d$-completeness for $\bold{GLC}$ is almost immediate, given that a $\bold{GLC}$ model is already a dynamic derivative model based on a scattered space via the standard up-set topology.%check commas

%\begin{prop}\label{glccomp} Suppose $\mathfrak{F}=\lb W,\sqsubset   ,g,V\rb $ is an irreflexive and weakly transitive %dynamic Kripke model and let $\tau$ be the upset topology on $W$. Then
%$$x \models_d \varphi\iff x \models \varphi.$$
%\end{prop}
%\begin{proof}
%We only need to prove the inductive step for $\varphi:=\square\psi$ as all the other steps are routine. 

%($\Rightarrow$) If $x \models_d\square\psi$ then there exists $U\in\tau$ such that $x\in U$ and $U\backslash\{x\}\models_d \psi$. By the induction hypothesis  $U\backslash\{x\}\models \psi$. By definition, $x\models \square\psi$ as $\uparrow(x)\backslash\{x\}\subseteq U$ and $R$ is irreflexive.

%$(\Leftarrow)$ If $x\models\square\psi$ then as $R$ is irreflexive and transitive then  $\uparrow \!(x)\backslash\{x\}\models \psi$. Let $U:=\uparrow\!x$. Then clearly $U$ is open and  by the induction hypothesis $U\backslash\{x\}\models_d \psi$. By definition it follows that $x\models_d\square\psi$.
%\end{proof}
%We can now prove the main theorem for $\bold{GLC}$.
\begin{restatable}{thm}{scattered}
$\bold{GLC}$ is the $d$-logic of all dynamic topological systems based on a scattered space. Furthermore, it has the finite model property for this class.
\end{restatable}

\begin{proof}
Soundness follows from Theorem~\ref{thmGLComp} and Lemma~\ref{lemCH}.
For completeness, suppose that $\nvdash_{\bold{GLC}}\varphi$.
Then by Theorem~\ref{wk4c} there is a finite model $\mathfrak M = \langle W,\sqsubset  ,g,\nu\rangle$ based on a transitive, irreflexive dynamic derivative frame such that $\mathfrak M \not\models\varphi$.
But by Lemmas~\ref{lemmDtau} and \ref{lemScatteredKrip}, setting $\tau=\tau_\sqsubset  $, we see that $d_\tau=d_\sqsubset$ and $\langle W,\tau,f\rangle$ is a scattered dynamic topological system, which is identical to $\mathfrak M$ as a dynamic derivative system and thus also refutes $\varphi$.
\end{proof}

In order to prove topological $d$-completeness for $\bold{wK4C}$, we must make a small modification to our Kripke frames.
Basically, weakly transitive frames are not necessarily isomorphic to any topological space, but weakly transitive {\em irreflexive} frames are.
Fortunately, any weakly transitive frame is bisimilar to an irreflexive frame.
The following construction is well known; see e.g.\ \cite{EsakiaAlgebra}.

\begin{defi}
Let $\mathfrak{F}=\langle W,\sqsubset   ,g\rangle$ be a $\bold{wK4C}$-frame and let $W^\mathrm i$ and $W^\mathrm r$ be the sets of irreflexive and reflexive points respectively. 
We define a new frame $\mathfrak{F}_\oplus=\langle W_\oplus,\sqsubset  _\oplus ,g_\oplus\rangle$, where
\begin{enumerate}

\item $W_\oplus = (W^{\rm i}\times \{0\})\cup (W^{\rm r}\times \{0,1\})$;

\item $(w,i) \sqsubset   _\oplus (v,j)$ iff $w\sqsubset   v$ and $(w,i) \neq (v,j)$;

\item $g_\oplus(w,i) = (g(w),0)$.

\end{enumerate}
\end{defi}

The following is standard \cite{EsakiaAlgebra} and easily verified.

\begin{prop}\label{propOplus}
If $\mathfrak{F}=\langle W,\sqsubset   ,g\rangle$ is any dynamic derivative frame, then $\mathfrak {F}_\oplus$ is an irreflexive dynamic derivative frame and $\pi\colon W_\oplus \to W $ given by $\pi(w,i)=w$ is a surjective, dynamic $p$-morphism.
\end{prop}

Proposition~\ref{propOplus} allows us to obtain topological completeness from Theorem~\ref{wk4c}, as Lemma~\ref{lemmDtau} tells us that irreflexive Kripke frames are essentially Cantor derivative spaces.
We thus obtain the following:

\begin{restatable}{thm}{topowkc}
	$\bold{wK4C}$ is the $d$-logic of all dynamic topological systems. Furthermore, it has the finite model property for this class.
\end{restatable}

\begin{proof}
Soundness follows from Lemma~\ref{lemCH}.
For completeness, suppose $\bold{wK4C}\nvdash\varphi$. Then by Theorem~\ref{wk4c} there exists a finite $\bold{wK4C}$-frame $\mathfrak{F}=\langle W,\sqsubset   , g\rangle$ such that $\mathfrak{F}\not\models\varphi$. But then, by Proposition~\ref{propOplus}, $\mathfrak{F}_\oplus = \langle W_\oplus,\sqsubset   _\oplus,g _\oplus \rangle$ is an irreflexive dynamic derivative system and $\pi \colon W_\oplus \to W$ is a surjective, dynamic $p$-morphism.
But by Lemma~\ref{lemmDtau} and by setting $\tau=\tau_\sqsubset  $ we see that $d_\sqsubset   = d_\tau$, so $\mathfrak{F}_\oplus$ can be seen as a dynamic topological system with the same semantics, also falsifying $\varphi$.
\end{proof}
Finally we turn our attention to $\bold{K4C}$.
Unlike the other two logics, $\bold{K4C}$ (or even $\bold{K4}$) does not have the topological finite model property, despite having the Kripke finite model property.
In the case of Aleksandrov spaces, the class $T_D$ is easy to describe.

\begin{lem}\label{lemmSingleton}
The Aleksandrov space of a $\bold{K4}$-frame $\mathfrak F = \langle W,\sqsubset   \rangle$ is $T_D$ if and only if $\sqsubset  $ is antisymmetric, in the sense that $w\sqsubset   v\sqsubset w$ implies $w=v$.%check coma
\end{lem}

If we moreover want the Kripke and the $d$-semantics to coincide on $\mathfrak F $, we need $\sqsubset$ to be irreflexive. Thus we wish to `unwind' $\mathfrak F$ to get rid of all non-trivial clusters.
The following construction achieves this.

\begin{defi}
Let $\mathfrak F = \langle W,\sqsubset   ,g\rangle$ be a dynamic $\bold{K4}$ frame.
We define a new frame
$\vec{ \mathfrak F} = \langle \vec W, \vec\sqsubset   , \vec g\rangle, $
where
\begin{itemize}

\item $\vec W$ is the set of all finite sequences $(w_0,\ldots,w_n)$, where $w_i \sqsubset   w_{i+1}$ for all $i<n$;

\item $\bold{w} \mathbin{\vec{\sqsubset}}  \bold{v}$ iff $\bold{w}$ is a strict initial segment of $\bold{v}$;

\item $\vec g (w_0,\ldots,w_n) $ is the subsequence of $(g(w_0), \ldots, g(w_n))$, obtained by deleting every entry that is equal to its immediate predecessor.

\end{itemize}
We then define a map $\pi\colon \vec W \to W$, where $\pi (w_0,\ldots,w_n) = w_n $.
\end{defi}
In the definition of $\vec g$, note that $g$ is only weakly monotonic, so it may be that, for instance, $g(w_0) = g(w_1)$; in this case, we include only one copy of $g(w_0)$ to ensure that $\vec g(\bold{w}) \in \vec W$.

\begin{restatable}{prop}{wind}\label{propUnwinding}
If $\mathfrak F = \langle W,\sqsubset   ,g\rangle$ is any dynamic $\bold{K4}$ frame, then $\vec {\mathfrak F}$ is an antisymmetric and irreflexive dynamic $\bold{K4}$ frame.
Moreover, $\pi \colon \vec W \to W$ is a surjective, dynamic $p$-morphism.
\end{restatable}

\begin{proof}
It is easy to see that $\vec {\mathfrak F}$ is antisymmetric and irreflexive. In order to show that $\vec{\mathfrak{F}}$ is a dynamic $\bold{K4}$ frame we need to show that it is transitive and $\vec g $ is a weakly monotonic function.
Suppose $\bold{w}\mathbin{\vec{\sqsubset}}   \bold{v}\mathbin{\vec{\sqsubset}}   \bold{u}$, then clearly since $\bold{w}$ is an initial segment of $\bold{v}$, it is also an initial segment of $\bold{u}$ and therefore $\bold{w}\mathbin{\vec{\sqsubset}}   \bold{u}$ and thus $\vec{\sqsubset}$ is transitive. The functionality of $\vec{g}$ follows from the functionality of $g$, and weak monotonicity can be checked from the definitions unsing the weak monotonicity of $g$. 
Finally, we show that $\pi$ is a surjective $p$-morphism. It is clearly surjective.
% To see that it is a $p$-morphism, first we note that ${\bf w}  \in \vec V(p) $ iff $w_n\in V(p)$ (by definition of $\vec V$) iff $\pi({\bf w}) \in V(p)$ (by definition of $\pi$).
Next, we note that if ${\bf w} \mathbin{\vec{\sqsubset}}     {\bf v}$ then transitivity plus $w_n=v_n$ yields $w_n \sqsubset    v_i$ for all $i\in(n,m]$, and in particular $w_n\sqsubset    v_m$, that is, $\pi( {\bf w}) \sqsubset \pi( {\bf v})$.
Moreover, if $\pi({\bf w}) \sqsubset    u $, then $w_n\sqsubset    u$ by definition of $\pi$, which means that ${\bf u} := (w_1,\ldots,w_n ,u) \in \vec W$ is such that ${\bf w} \mathrel{\mathbin{\vec{\sqsubset}}   } {\bf u}$ and $\pi({\bf u}) = u$.
Furthermore, $\pi(\vec{g}(w_0,\dots, w_n))=g(w_n)=g(\pi(w_0,\dots w_n))$ and thus $\pi$ is a $p$-morphism.
\end{proof}

Similar constructions have already appeared in e.g.\ \cite{konev}.
From this we obtain the following:

\begin{restatable}{thm}{topokc}
$\bold{K4C}$ is the $d$-logic of all dynamic topological systems based on a $T_D$ space, as well as the $d$-logic of all dynamic topological systems based on an Aleksandrov $T_D$ space.
\end{restatable}

\begin{proof}
Soundness follows from soundness for $\bold{wK4C}$ and Lemma~\ref{lemmSingleton}.
For completeness, if $\bold{K4C} \not \vdash \varphi$, by Theorem~\ref{wk4c} there is a $\bold{K4C}$ model $\mathfrak M = \langle W,\sqsubset   ,g,V \rangle$ falsifying $\varphi$.
Let $\mathfrak F = \langle W,\sqsubset   ,g  \rangle$.
But by Proposition~\ref{propUnwinding}, $\pi\colon \vec W\to W$ is a surjective $p$-morphism, which implies that $\varphi$ is not valid on $ \vec{\mathfrak F}$.
By Proposition~\ref{propUnwinding}, $\vec\sqsubset   $ is antisymmetric, hence by Lemma~\ref{lemmSingleton}, the Aleksandrov topology of $ \vec{\mathfrak F}$ is $T_D$; moreover, $\vec\sqsubset    $ is also irreflexive, so we see that $\varphi$ is falsified on this $T_D$ Aleksandrov space as well.
\end{proof}
\section{Invertible Systems}

Recall that if $\lb X,\tau\rb$ is a topological space and $f\colon X\to X$ is a function, then $f$ is a {\em homeomorphism} if it is a bijection and both $f$ and $f^{-1}$ are continuous.

Unlike in the continuous case, for logics with homeomorphisms, every formula is equivalent to a formula where all occurrences of $\bc$ are applied to atoms.
More formally, we say that a formula is in {\em $\bc$-normal form} if it is of the form $\varphi(\bc^{k_0}p_0,\ldots,\bc^{k_n}p_n)$, where $\varphi(p_0,\ldots,p_n)$ does not contain $\bc$, the $p_i$'s are variables, and the $k_i$'s are natural numbers; in other words, $\varphi$ is in $\bc$=normal form if the only occurrences of $\bc$ are of the form $\bc^{k}p$ where $p$ is atomic.
It is readily observed that the axiom {\rm H}, together with ${\rm Next_\neg}$ and ${\rm Next_\wedge}$, allows us to `push' all instances of $\bc$ to the propositional level. Thus we obtain the following useful representation lemma:

\begin{restatable}{lem}{bcnormal}\label{lemmBCNormal}
For every logic $\Lambda$ and every formula $\varphi$, there is a formula $\varphi'$ in $\bc$-normal form such that $\Lambda{\bf H}\vdash \varphi\leftrightarrow \varphi'$.%check comma
\end{restatable}

\begin{proof}
This is a standard proof by induction on the complexity of $\varphi$. In particular, note that the axioms $\square\bc p\rightarrow \bc \square p$ and $\bc\square p\rightarrow\square \bc p$ permit the interchangeability between $\bc$ and $\square$, thus allow us to push $\bc$ towards the atomic part of the formula. 
\end{proof}

As we shall see shortly, $\Lambda{\bf H}$ inherits completeness and the finite model property almost immediately from $\Lambda$.
If $\mathfrak X = \langle X, \rho_X\rangle$ and $\mathfrak Y = \langle Y,\rho_Y\rangle$ are derivative spaces, then the {\em topological sum} is defined as $\mathfrak X\oplus \mathfrak Y = (X\oplus Y,\rho_X\oplus \rho_Y)$, where $X\oplus Y$ is the disjoint union of the two sets, and $\rho_X\oplus \rho_Y$ is given by $(\rho_X\oplus \rho_Y) A = \rho_X(A\cap X) \cup \rho_Y(A\cap Y)$.
We say that a class $\Omega$ of derivative spaces is {\em closed under sums} if whenever $\mathfrak A,\mathfrak B\in \Omega$, it follows that $\mathfrak A \oplus \mathfrak B \in \Omega$.

Given a derivative space $\mathfrak A = \langle A,\rho\rangle$, we may write $\mathfrak A^n$ instead of $\mathfrak A\oplus \mathfrak A\oplus \ldots \oplus \mathfrak A$ ($n$ times), and if $\mathfrak A$ has domain $A$, we may identify the domain of $\mathfrak A^n$ with $A^n = A\times \{0,\ldots,n-1\}$.
It should be clear that if $\mathfrak A\in \Omega$ and $\Omega$ is closed under sums, then $\mathfrak A^n\in\Omega$.
Let $(m)_n$ denote the remainder of $m$ modulo $n$.
We then define a dynamical structure $\mathfrak A^{(n)} = (\mathfrak A^n,f)$, where $f\colon A^n\to A^n$ is given by $f(w,i) = (w,(i+1)_n)$.

\begin{restatable}{lem}{invert}\label{lemmIsInvert}
Let $\Omega$ be a class of derivative spaces closed under sums. Then, if $\mathfrak A\in\Omega$, it follows that $\mathfrak A^{(n)}$ is an invertible dynamic derivative system based on an element of $\Omega$.
\end{restatable}

\begin{proof}
Suppose $\mathfrak A = \langle X,d\rangle \in\Omega$ is a derivative space.
Let $\mathfrak A^{(n)}=(X^n,d',f)$. Clearly $f$ is a bijection as it defines disjoint cycles.
If $Y\subseteq X^n$, let $Y_i = \{x \in X: (x,i) \in Y\}$.
Then, $d' f^{-1} Y = \bigcup_{i<n} (d Y_i,(i-1)_n) = f^{-1} \bigcup_{i<n} (d Y_i,i) = f^{-1} d' Y$.
So, $ \mathfrak A^{(n)}$ is a dynamic derivative frame. 
\end{proof}
For our proof of completeness, we need the notion of {\em extended valuation.}

\begin{defi}
Let $\mathfrak A = \langle A,\rho \rangle$ be a derivative space.
Let $\mathsf{PV}^\bc$ be the set of all expressions $\bc^i p$, where $i\in\mathbb N$ and $p$ is a propositional variable.
An {\em extended valuation} on $\mathfrak A$ is a relation $\nu \subseteq \mathsf{PV}^\bc \times A $.

If $\nu$ is an extended valuation on $\mathfrak A$, we define a valuation on $\mathfrak A^{(n)}$ so that for any variable $p$ and $(w,i)\in A^n$, $(w,i)\in \nu^{(n)}(p)$ if and only if $ w \in \nu (\bc^i p) $.
\end{defi}

\begin{restatable}{lem}{npreserve}\label{lemmNPreserve}
Let $\mathfrak A = \langle A,\rho\rangle$ be any derivative space and $ \nu $ be any extended valuation on $\mathfrak A$.
Let $w\in A$, $i<n$, and $\varphi$ be any formula in $\bc$-normal form which has $\bc$-depth less than $n-i$.
Then, $\langle \mathfrak A^{(n)}, \nu ^{(n)}\rangle ,(w,i) \models \varphi$ if and only if
$\langle \mathfrak A , \nu \rangle , w  \models \varphi.$
\end{restatable}

\begin{proof}
We start with the right to left direction and we only show the case where $\varphi=\bc \psi $. Since $\varphi$ is in $\bc$-normal form then it is of the form $\bc^j p$. Since $j<n-i$ and by definition $(w,j)\in  \nu ^{(n)} p$, then $\langle \mathfrak A^{(n)}, \nu ^{(n)}\rangle ,(w,i) \models \bc^jp$. The other direction follows similarly.
The rest of the modalities and connectives follow by a standard induction on the build of $\varphi$.  
\end{proof}
It then readily follows that $\varphi$ is consistent if and only if $\varphi'$ is consistent, which implies that $\varphi'$ is satisfiable on $\Omega$. 
This is equivalent to $\varphi$ being satisfiable on the class of invertible systems based on $\Omega$.
From this, we obtain the following general result.  

\begin{restatable}{thm}{topoh}
Let $\Lambda$ be complete for a class $\Omega$ of derivative spaces.
Then, if $\Omega$ is closed under sums, it follows that $\Lambda{\bf H}$ is complete for the class of invertible systems based on $\Omega$.
\end{restatable}

\begin{proof}
Assume that $\Lambda$ and $\Omega$ are as in the statement of the theorem.
Let $\varphi$ be any formula, and using Lemma~\ref{lemmBCNormal}, let $\varphi'$ be a $\Lambda{\bf H}$-equivalent formula in $\bc$-normal form.
Let $n$ be the $\bc$-depth of $\varphi'$.
If $\varphi$ is consistent, then $\varphi'$ is consistent over $\Lambda$.
By the assumption, $\varphi'$ is satisfied on some $\mathfrak A =\langle A,\rho\rangle \in\Omega$.
Suppose it is satisfied on a point $w_*\in A$.
By Lemma~\ref{lemmIsInvert}, $\mathfrak A^{(n)}$ is an invertible system based on an element of $\Omega$, which by Lemma~\ref{lemmNPreserve} satisfies $\varphi'$ on $(w_*,0)$.
But $\Lambda{\bf H}$ is sound for the class of invertible systems based on $\Omega$, hence $\varphi$ is also satisfied on $(w_*,0)$.
\end{proof}
\begin{cor}\

\begin{enumerate}

\item $\bf wK4H$ is sound and complete for
\begin{enumerate}

\item The class of all finite invertible dynamic $\bf wK4$ frames.

\item The class of all finite invertible dynamic topological systems with Cantor derivative.

\end{enumerate}

\item $\bf K4H$ is sound and complete for
\begin{enumerate}

\item The class of all finite invertible dynamic $\bf K4$ frames.

\item The class of all invertible, $T_D$ dynamic topological systems with Cantor derivative.

\end{enumerate}

\item $\bf GLH$ is sound and complete for
\begin{enumerate}

\item The class of all finite invertible dynamic $\bf GL$ frames.

\item The class of all finite invertible, scattered dynamic topological systems with Cantor derivative.

\end{enumerate}

\end{enumerate}

\end{cor}

\color{black}

\section{Conclusion}

We have recast dynamic topological logic in the more general setting of derivative spaces and established the seminal results for the $\{\dd,\bc\}$-fragment for the variants of the standard logics with the Cantor derivative in place of the topological closure.
Semantics on the Cantor derivative give rise to a richer family of modal logics than their counterparts based on closure. This is evident in the distinction between e.g.\ the logics $\mathbf{wK4C}$ and $\mathbf{K4C}$, both of which collapse to $\mathbf{S4C}$ when replaced by their closure-based counterparts.
This line of research goes hand in hand with recent trends that consider the Cantor derivative as the basis of topological semantics \cite{FernandezIliev,Kudinov}.

There are many natural problems that remain open.
The logic $\mathbf{S4C}$ is complete for the Euclidean plane.
In the context of $d$-semantics, the logic of the Euclidean plane is a strict extension of $\mathbf{K4}$, given that punctured neighbourhoods are {\em connected} in the sense that they cannot be split into two disjoint, non-empty open sets.
Thus one should not expect $\mathbf{K4C}$ to be complete for the plane.
This raises the question: what is the dynamic $d$-logic of Euclidean spaces in general, and of the plane in particular?

The $d$-semantics also poses new lines of inquiry with respect to the class of functions considered.
We have discussed continuous functions and homeomorphisms.
Artemov et al.\ \cite{artemov} also considered arbitrary functions, and we expect that the techniques used by them could be modified without much issue for $d$-semantics.
On the other hand, in the setting of closure-based logics, the logic of spaces with continuous, open maps that are not necessarily bijective coincides with the logic of spaces with a homeomorphism.
This is no longer true in the $d$-semantics setting, as the validity of the $\rm H$ axiom requires injectivity.
Along these lines, the $d$-logic of {\em immersions} (i.e.\ continuous, injective functions) would validate the original continuity axiom $\bc\dn p\to\dn\bc p$.
Thus there are several classes of dynamical systems whose closure-based logics coincide, but are split by $d$-semantics.

Finally, there is the issue of extending our language to the trimodal language with the `henceforth' operator.
It is possible that the $d$-logic of all dynamic topological systems may be axiomatised using the tangled derivative, much as the tangled closure was used to provide an axiomatisation of the closure-based $\mathbf{DTL}$.
However, given that the tangled derivative is made trivial on scattered spaces, we conjecture that the trimodal $d$-logic based on this class has a natural, finite axiomatisation.
The work presented here is an important first step towards proving this.

\bibliographystyle{alphaurl}

\bibliography{bibliography}
% \end{document}

 %%%%BOOKMARK

\end{document}